\documentclass{article}
\usepackage{amssymb,amscd,amsmath,amsthm,xcolor}
\usepackage{graphics}
\usepackage[dvips]{graphicx}
\usepackage{hyperref}
\usepackage[noabbrev,capitalise]{cleveref}
\usepackage{url}
\usepackage{authblk}

\usepackage{mathtools}

\def\NN{\mathbb{N}}
\usepackage{helvet}

\def\PP{\mathbb{P}}
\def\QQ{\mathbb{Q}}
\def\F{\mathcal{F}}
\def\U{\mathcal{U}}
\def\C{\mathcal{C}}
\def\I{\mathcal{I}}
\def\D{\mathcal{D}}

\def\A{\mathcal{A}}
\def\P{\mathcal{P}}
\def\M{\mathcal{M}}
\def\Nc{\mathcal{N}}
\def\PP{\mathbb{P}}
\def\Q{\mathcal{Q}}
\def\vr{$\vec{R}$}
\renewcommand{\S}{\mathcal{S}}
\newcommand{\R}{\mathcal{R}}
\usepackage{algorithm2e}
\RestyleAlgo{ruled}
\usepackage{babel}[french]
\usepackage{tikz}
\renewcommand{\L}[0]{\mathcal{L}}

\usepackage{xcolor} % pour les métacommentaires

\newcommand{\qvdash}{\operatorname{{?}{\vdash}}}
\newcommand{\nqvdash}{\operatorname{{?}{\nvdash}}}
\newcommand{\uh}{\upharpoonright}

\newcommand{\RCA}{\mathsf{RCA}}
\newcommand{\WKL}{\mathsf{WKL}}

\newcommand{\EM}{\mathsf{EM}}
\newcommand{\RT}{\mathsf{RT}}
\newcommand{\MCO}{$\M$-cohesive }
\newcommand{\UCM}{\U_C^\M}
\newcommand{\UCNMN}{\U_{C_n}^{\M_n}}

\newcommand{\seq}[1]{$({#1}_n)_{n \in \omega}$}
\def\qt#1{``#1''}%
\newcommand{\Formula}{\text{Form}}

\newtheorem{theorem}{Theorem}
\numberwithin{theorem}{section}
\newtheorem{maintheorem}[theorem]{Main Theorem}
\newtheorem{lemma}[theorem]{Lemma}
\newtheorem{proposition}[theorem]{Proposition}
\theoremstyle{definition}
\newtheorem{definition}[theorem]{Definition}

\newtheorem{remark}[theorem]{Remark}
\newtheorem{statement}[theorem]{Statement}

\makeatletter
\newtheorem*{rep@theorem}{\rep@title}
\newcommand{\newreptheorem}[2]{%
\newenvironment{rep#1}[1]{%
 \def\rep@title{#2 \ref{##1}}%
 \begin{rep@theorem}}%
 {\end{rep@theorem}}}
\makeatother

\newreptheorem{maintheorem}{Main Theorem}

\title{The weakness of the Erd\H{o}s-Moser theorem\\ under arithmetic reductions}
\author[1]{Ludovic Levy Patey}
\affil[1]{CNRS, IMJ-PRG\\
   Universit\'e Paris Cit\'e\\
   Paris, France} 
\author[2]{Ahmed Mimouni}
\affil[2]{Univ Paris Est Creteil, LACL, F-94010 Creteil, France} 
\begin{document}
\maketitle

\begin{abstract}
The Erd\H{o}s-Moser theorem ($\EM$) says that every infinite tournament admits an infinite transitive subtournament.
We study the computational behavior of the Erd\H{o}s-Moser theorem with respect to the arithmetic hierarchy,
and prove that $\Delta^0_n$ instances of~$\EM$ admit low${}_{n+1}$ solutions for every~$n \geq 1$, and that if a set~$B$ is not arithmetical,  then every instance of~$\EM$ admits a solution relative to which~$B$ is still not arithmetical. We also provide a level-wise refinement of this theorem.
These results are part of a larger program of computational study of combinatorial theorems in Reverse Mathematics.
\end{abstract}

\section{Introduction}

We conduct a computational study of the Erd\H{o}s-Moser theorem, an infinitary statement from graph theory. 
A \emph{tournament} on a domain~$D \subseteq \NN$ is an irreflexive binary relation $R \subseteq D^2$ such that for every~$a, b \in D$ with~$a \neq b$, exactly one of $R(a, b)$ and $R(b, a)$ holds. A tournament $R$ is \emph{transitive} if for every~$a, b, c \in D$, if $R(a, b)$ and $R(b, c)$ then $R(a, c)$. A subtournament of~$R$ is the restriction of~$R$ to a subdomain~$H \subseteq D^2$. We identify subtournaments with their domains.
The following statement is known as the Erd\H{o}s-Moser theorem, and is an infinitary version of some theorem by Erd\H{o}s and Moser~\cite{erdos1964representation}.

\begin{statement}[Erd\H{o}s-Moser theorem]
$\EM$ is the statement \qt{Every infinite tournament admits an infinite transitive subtournament.}
\end{statement}

The Erd\H{o}s-Moser theorem easily follows from the celebrated Ramsey theorem for pairs.
Given a set~$X \subseteq \NN$ and some integer~$n \in \NN$, we let $[X]^n$ denote the set of all unordered $n$-tuples over~$X$. 
Given a coloring $f : [\NN]^n \to k$, a set~$H \subseteq \NN$ is \emph{$f$-homogeneous (for color~$i < k$)}  if $f(\sigma) = i$ for every~$\sigma \in [\NN]^n$. 

\begin{statement}[Ramsey's theorem]
Given $n, k \in \NN$, $\RT^n_k$ is the statement \qt{Every coloring $f : [\NN]^n \to k$ admits an infinite $f$-homogeneous set}.
\end{statement}

There exists a one-to-one correspondence between a tournament~$R$ and a coloring $f : [\NN]^2 \to 2$, by letting $f(\{x, y\}) = 1$ if $(R(x, y) \leftrightarrow x <_\NN y)$. $\EM$ can the restated as \qt{for every coloring $f : [\NN]^2 \to 2$, there exists an infinite \emph{$f$-transitive} subset~$H$, that is, for every~$x, y, z \in H$ such that $x < y < z$ and every~$i < 2$, then if $f(\{x, y\}) = f(\{y, z\}) = i$, we have $f(\{x, z\}) = i$.}
Since any $f$-homogeneous set is $f$-transitive, the Erd\H{o}s-Moser theorem can be considered as a particular case of Ramsey's theorem for pairs.

\subsection{$\EM$ and $\RT^2_2$ in Reverse Mathematics}

Both the Erd\H{o}s-Moser theorem and Ramsey's theorem for pairs have been extensively studied in Reverse Mathematics, both from a computational and a proof-theoretic viewpoint. See Hirschfeldt~\cite{hirschfeldt2017slicing} for an introduction to the reverse mathematics of combinatorial principles.

From many perspectives, $\EM$ is very close to $\RT^2_2$. The combinatorics are very similar, and the Erd\H{o}s-M\H{o}ser theorem can be considered as a disjunction-free version of Ramsey's theorem for pairs. These similarities in combinatorics have many consequences in Reverse Mathematics. Jockusch~\cite{jockusch1972ramsey} proved that every computable instance of~$\RT^2_2$ admits a $\Pi^0_2$ solution, while there exists a computable instance of~$\RT^2_2$ with no $\Sigma^0_2$ solution. These bounds are the same for the Erd\H{o}s-Moser theorem.
On the proof-theoretic side, the first-order part of Ramsey's theorem for pairs and the Erd\H{o}s-Moser theorem are known to coincide~\cite{chong2021pi11}. More generally, most of the known statements implied by $\RT^2_2$ are already known to follow from~$\EM$ over~$\RCA_0$, the base theory of Reverse Mathematics. Whether $\EM$ implies~$\RT^2_2$ was open for a long time, before Lerman, Solomon and Towsner~\cite{lerman2013separating} answered it negatively.

When considering non-computable instances, the behaviors of $\RT^2_2$ and $\EM$ turn out to be very dramatically different. For every function $g : \NN \to \NN$, there exists a coloring $f : [\NN]^2 \to 2$ such that every infinite $f$-homogeneous set computes a function~$h$ dominating $g$, that is, $\forall x(h(x) \geq g(x))$. Indeed, simply take $f(x, y) = 1$ iff $g(x) < y$. Thus, by a theorem of Slaman and Groszek~\cite{slaman2007moduli}, there exists a (non-computable) instance of $\RT^2_2$ such that every solution computes every hyperarithmetic (or equivalently $\Delta^1_1$) set. On the other hand, Patey and Wang (both unpublished) independently proved that for every non-computable set~$B$ and every instance of~$\EM$, there exists a solution which does not compute~$B$. This property of~$\EM$ is shared with the infinite pigeonhole principle ($\RT^1_2$). Indeed, Dzhafarov and Jockusch~\cite{jockusch2009ramsey} proved that for every non-computable set~$B$ and every set~$A$, there is an infinite subset~$H$ of~$A$ or $\overline{A}$ which does not compute~$B$.

\subsection{$\EM$ and $\RT^1_2$ under stronger reductions}

As mentioned above, when considering non-computable instances, the Erd\H{o}s-Moser theorem seems to have closer behavior to the pigeonhole principle than to Ramsey's theorem for pairs. In a series of papers, Monin and Patey~\cite{monin2019pigeons,monin2021weakness,monin2021models} developed a framework to control iterated jumps of solutions to the pigeonhole principle. They proved in particular the following three facts (see~\cite{monin2021weakness}):
\begin{itemize}
    \item If~$B$ is not arithmetic (resp.\ hyperarithmetic), then for every set~$A$, there is an infinite subset~$H$ of~$A$ or $\overline{A}$ such that~$B$ is not $A$-arithmetic (resp.\ $A$-hyperarithmetic).
    \item If $B$ is not $\Sigma^0_n$ (resp.\ $\Delta^0_n)$, then for every set~$A$, there is an infinite subset~$H$ of~$A$ or $\overline{A}$ such that~$B$ is not $\Sigma^0_n(A)$ (resp.\ $\Delta^0_n(A)$).
    \item For every $\Delta^0_n$ set~$A$, there is an infinite subset~$H$ of~$A$ or $\overline{A}$ of low${}_n$ degree.
\end{itemize}

In this article, we prove the following three theorems:

\begin{maintheorem}\label[theorem]{thm:arithmetic-main}
If $B$ is not arithmetic, then for every tournament~$T$, there is an infinite transitive subtournament~$H$ such that $B$ is not $H$-arithmetic.
\end{maintheorem}

The generalization to the hyperarithmetic hierarchy is not proven, but the authors believe that it holds with the same proof \emph{mutatis mutandis}. Like for the pigeonhole principle, a layerwise version of the previous theorem holds:

\begin{maintheorem}\label[theorem]{thm:layerwise-main}
Fix $n \geq 1$. If $B$ is not $\Sigma^0_n$, then for every tournament~$T$, there is an infinite transitive subtournament~$H$ such that $B$ is not $\Sigma^0_n(H)$.
\end{maintheorem}

The case $n = 1$ was already proven independently by the first author and Wang (unpublished).
The statement where $\Sigma^0_n$ is replaced by $\Delta^0_n$ directly follows from \Cref{thm:layerwise-main}
and Post's theorem. Indeed, if $B$ is not $\Delta^0_n$, then by Post's theorem, either it or its complement is not $\Sigma^0_n$. Then apply \Cref{thm:layerwise-main} to conclude.

\begin{maintheorem}\label[theorem]{thm:effective-main}
Fix $n \geq 1$. Every $\Delta^0_n$ tournament~$T$ has an infinite transitive subtournament of low${}_{n+1}$ degree.
\end{maintheorem}

The case $n = 1$ follows from the same statement for Ramsey's theorem for pairs, proven by Cholak, Jockusch and Slaman~\cite{cholak_jockusch_slaman_2001}. On the other hand, as explained, the statement for Ramsey's theorem for pairs fails for~$n > 1$, since there exists a $\Delta^0_2$ instance of $\RT^2_2$ such that every solution computes~$\emptyset'$.

Besides the intrinsic interest of these theorems, they are also motivated by the more general program of development of good iterated jump control of combinatorial theorems, and in particular by the goal to prove the strictness of the free set, thin set and rainbow ramsey hierarchies. See Monin and Patey~\cite[Section 1.2]{monin2021weakness} for a discussion on the subject. The Erd\H{o}s-Moser is the first statement about colorings of pairs which is known to admit a good iterated jump control.

\subsection{Definition and notation}

A \emph{binary string} is an ordered tuple of bits $a_0, \dots, a_{n-1} \in \{0, 1\}$.
The empty string is written $\epsilon$. A \emph{binary sequence} (or a \emph{real}) is an infinite listing of bits $a_0, a_1, \dots$.
Given $s \in \omega$,
$2^s$ is the set of binary strings of length $s$ and
$2^{<s}$ is the set of binary strings of length $<s$. Accordingly,
$2^{<\omega}$ is the set of binary strings
and $2^{\omega}$ is the set of binary sequences.
Given a string $\sigma \in 2^{<\omega}$, we use $|\sigma|$ to denote its length.
Given two strings $\sigma, \tau \in 2^{<\omega}$, $\sigma$ is a \emph{prefix}
of $\tau$ (written $\sigma \preceq \tau$) if there exists a string $\rho \in 2^{<\omega}$
such that $\sigma^\frown \rho = \tau$. Given a sequence $X$, we write $\sigma \prec X$ if
$\sigma = X \uh n$ for some $n \in \omega$.
A binary string $\sigma$ can be interpreted as a finite set $F_\sigma = \{ x < |\sigma| : \sigma(x) = 1 \}$. We write $\sigma \subseteq \tau$ for $F_\sigma \subseteq F_\tau$.
We write $\#\sigma$ for the size of $F_\sigma$. Given two strings $\sigma$ and $\tau$, we let $\sigma \cup \tau$ be the unique string $\rho$ of length $\max(|\sigma|, |\tau|)$ such that $F_\rho = F_\sigma \cup F_\tau$. 

A \emph{binary tree} is a set of binary strings $T \subseteq 2^{<\omega}$ which is closed downward under the prefix relation. A \emph{path} through $T$ is a binary sequence $P \in 2^\omega$ such that every initial segment belongs to $T$.

A \emph{Turing ideal} $\I$ is a collection of sets which is closed downward under the Turing reduction and closed under the effective join, that is, $(\forall X \in \I)(\forall Y \leq_T X) Y \in \I$ and $(\forall X, Y \in \I) X \oplus Y \in \I$, where $X \oplus Y = \{ 2n : n \in X \} \cup \{ 2n+1 : n \in Y \}$. A \emph{Scott set} is a Turing ideal $\I$ such that every infinite binary tree $T \in \I$ has a path in $\I$. In other words, a Scott set is the second-order part of an $\omega$-model of $\RCA_0 + \WKL$.
A Turing ideal $\M$ is \emph{countable coded} by a set $X$
if $\M = \{ X_n : n \in \omega \}$ with $X = \bigoplus_{n \in \omega} X_n$.
Given $n \geq 1$, a formula is $\Sigma^0_n(\M)$ (resp.\ $\Pi^0_n(\M)$) if it is $\Sigma^0_n(X)$ (resp.\ $\Pi^0_n(X)$) for some $X \in \M$.

Given two sets $A$ and $B$, we denote by $A < B$ the formula
$(\forall x \in A)(\forall y \in B)[x < y]$.
We write $A \subseteq^{*} B$ to mean that $A - B$ is finite, that is,
$(\exists n)(\forall a \in A)(a \not \in B \rightarrow a < n)$.
A \emph{$k$-cover} of a set $X$ is a sequence of sets $Y_0, \dots, Y_{k-1}$ such that $X \subseteq Y_0 \cup \dots \cup Y_{k-1}$.

\subsection{Organization of this paper}

In \Cref{sect:big-picture}, we try to give an overview of the forcing construction, by explaining in \Cref{sect:picture-forcing-question} the importance of the so-called \qt{forcing question}, then  diving in \Cref{sect:combinatorics-em} into the combinatorics of~$\EM$, especially explaining the role of the infinite pigeonhole principle as a warrant of extendibility for the Erd\H{o}s-Moser theorem, and then explaining in \Cref{sect:iterated-jump-picture} the issues raised when trying to control iterated jumps of solutions with variants of Mathias forcing.

In \Cref{sect:largeness-pr}, we restate the main properties of partition regular and large classes, studied in Monin and Patey~\cite{monin2019pigeons,monin2021weakness}. In particular, we define the notions of cohesive and minimal classes in~\Cref{sect:minimal-cohesive-classes}, which play an essential role to maintain the compatibility of large classes between different levels of the iterated jump control. Last, we restate in \Cref{sect:uc-sequence} the existence of a hierarchy of Scott sets and of cohesive classes, which play the role of a spine for the main notion of forcing.

\Cref{sect:forcing-framework} is dedicated to the development of the main forcing framework, by defining its conditions, the forcing relation, and a forcing question. This framework is applied in various contexts, to prove strong cone avoidance of~$\EM$ for arithmetic reductions in \Cref{sect:strong-avoidance-arith}, a layerwise version of this strong cone avoidance for $\Sigma^0_n$ operators in~\Cref{sect:layerwise-avoidance}, and prove the existence of low${}_n$ solutions through an effectivization of the construction in~\Cref{sect:effective-constructions}.

\section{The big picture}\label[section]{sect:big-picture}

The techniques used in this article are rather sophisticated with many technical subtleties, and it may be quite hard to have the big picture. In this section, we describe the general forcing argument used to prove our main theorems, and highlight a few technical difficulties justifying the design choice of our notion of forcing.

\subsection{Forcing question}\label[section]{sect:picture-forcing-question}

The three main theorems are related, in that they involve very similar techniques of iterated jump control. Indeed, in each case, it consists of constructing a solution whose $\Sigma^0_n$ properties resemble the ones of the ground model. For this, one tries to translate $\Sigma^0_n(G)$ formulas relative to the generic object constructed~$G$ into absolute $\Sigma^0_n$ formulas. In set-theoretic forcing, this is achieved through the forcing relation, whose definition must be sufficiently simple (in terms of definitional complexity) to make the new model inherit properties of the ground model. In computability theory, the situation is slightly different, and the simplicity of the forcing relation is less important than the one of the so-called \emph{forcing question}. 

In what follows, a \emph{notion of forcing} is a partial order $(\PP, \leq)$
such that every sufficiently generic filter~$\F \subseteq \PP$ induces a set~$G_\F \subseteq \NN$.
Every notion of forcing is equipped with a forcing relation, written $\Vdash$, between the set of conditions $\PP$ and the set of arithmetic formulas $\Formula[G]$ with a set parameter~$G$ denoting the generic object constructed.

\begin{definition}
Fix a notion of forcing $(\PP, \leq)$. 
A \emph{forcing question} is a relation $\qvdash$ over $\PP \times \Formula[G]$ such that, for every~$c \in \PP$ and $\varphi(G) \in \Formula[G]$,
\begin{enumerate}
	\item[(1)] If $c \qvdash \varphi(G)$, then there is an extension $d \leq c$ such that $d \Vdash \varphi(G)$;
	\item[(2)] If $c \nqvdash \varphi(G)$, then there is an extension $d \leq c$ such that $d \Vdash \neg \varphi(G)$;
\end{enumerate}	
\end{definition}

The notion of forcing question is not canonical, and a single notion of forcing might have many candidate forcing questions. On the other hand, many computational properties of the generic object~$G$ might be directly derived from the existence of a forcing question with sufficiently nice definitional properties.
Consider for example the following property:

\begin{definition}\label[definition]{def-uniformly-preserving}
Fix a notion of forcing $(\PP, \leq)$ and some~$n \in \NN$.
A forcing question $\qvdash$ is \emph{uniformly $\Sigma^0_n$-preserving} if for every~$c \in \PP$
and every uniform sequence of $\Sigma^0_n$ formulas $\varphi_0(G), \varphi_1(G), \dots$, the sequence $c \qvdash \varphi_0(G), c \qvdash \varphi_1(G), \dots$ is uniformly $\Sigma^0_n$.
\end{definition}

The following proposition is at the heart of our forcing construction. It was used by Wang~\cite[Proposition 3.1, Proposition 3.4, Theorem 3.6]{wang2016definability}, where the author showed for each notion of forcing the existence of a uniformly $\Sigma^0_n$-preserving forcing question, without naming explicitly this concept.

\begin{proposition}\label[proposition]{prop-uniform-preservation-sigman}
Let~$(\PP, \leq)$ be a notion of forcing with a uniformly $\Sigma^0_n$-preserving forcing question.
Then for every non-$\Sigma^0_n$ set~$C$ and every sufficiently generic set~$G$ for this notion of forcing, $C$ is not $\Sigma^0_n(G)$.
\end{proposition}
\begin{proof}
Fix a $\Sigma^0_n$ formula $\varphi(G, x)$ with one free first-order variable~$x$. Let~$D_\varphi \subseteq \PP$ be the set of conditions~$c$ such that either $c \Vdash \varphi(G, a)$ for some~$a \not \in C$, or $c \Vdash \neg \varphi(G, a)$ for some~$a \in C$. Let us show that~$D_\varphi$ is dense.
Given a condition~$c \in \PP$, let~$W = \{ a \in \NN : c \qvdash \varphi(G, a) \}$. Since the forcing question is uniformly $\Sigma^0_n$-preserving, then the set~$W$ is $\Sigma^0_n$, hence~$C \neq W$. Let~$a \in C \Delta W$, the symmetric difference of~$C$ and $W$.
\begin{itemize}
	\item If~$a \in W \setminus C$, then by definition, $c \qvdash \varphi(G, a)$, so by property (1) of the forcing question, there is an extension~$d \leq c$ such that $d \Vdash \varphi(G, a)$.
	\item If~$a \in C \setminus W$, then by definition, $c \nqvdash \varphi(G, a)$. By property (2) of the forcing question, $c \qvdash \neg \varphi(G, a)$, and by property (1), there is an extension $d \leq c$ such that $d \Vdash \neg \varphi(G, a)$.
\end{itemize}
In both cases, $d \in D_\varphi$, so $D_\varphi$ is dense.
If $\F$ is a sufficiently generic filter, it will intersect $D_\varphi$ for every $\Sigma^0_n$ formula $\varphi(G, x)$, hence, letting~$G$ be the set induced by~$\F$, $C$ will not be~$\Sigma^0_n(G)$.
\end{proof}

The construction of low${}_n$ solutions are often effectivizations of the forcing argument,
either by constructing a $\emptyset^{(n)}$-computable filter sufficiently generic for deciding $\Sigma^0_n(G)$ formulas, or by constructing, with any PA degree~$P$ over~$\emptyset^{(n-1)}$, a $P$-computable filter~$G_\F$ sufficiently generic for deciding $\Sigma^0_{n-1}(G)$ formulas. In the latter case, using the low basis theorem relativized to~$\emptyset^{(n-1)}$, there exists such a PA degree~$P$ over~$\emptyset^{(n-1)}$ such that $P' \leq_T \emptyset^{(n)}$, thus $\Sigma^0_n$ properties of~$G_\F$ can be decided thanks to~$\emptyset^{(n)}$.

\subsection{Combinatorics of~$\EM$}\label[section]{sect:combinatorics-em}

In order to understand the design of the notion of forcing for this article, it is important to get familiar with the combinatorics of the Erd\H{o}s-Moser theorem. Lerman, Solomon and Towsner~\cite{lerman2013separating} analyzed the basic combinatorial ideas essential to the computable study of the theorem.

A transitive tournament $T$ over a domain~$A$ can be seen as a linear order $(A, \leq_T)$ 
defined by~$a \leq_T b$ iff $a = b$ or $T(a, b)$. This interpretation should be kept in mind throughout the article. For convenience, we shall always consider that the tournament contains two end-points $-\infty$ and $+\infty$, that is, such that $T(-\infty, x)$ and $T(x, +\infty)$ holds for every~$x$.

\begin{definition}[\cite{lerman2013separating}]
Fix a tournament~$T$ over a domain~$A$.
\begin{itemize}
    \item[(1)] The \emph{interval} $(a, b)$ between $a, b \in A \cup \{-\infty, +\infty\}$ is the set of points $x \in A$ such that $T(a, x)$ and $T(x, b)$ holds. 
    \item[(2)] Given a finite $T$-transitive subset~$F \subseteq A$ and $a, b \in F \cup \{-\infty, +\infty\}$, the interval $(a, b)$ is \emph{minimal} in~$F$ if $(a, b) \cap F = \emptyset$.
\end{itemize}
\end{definition}

Fix a tournament~$T$ over~$\omega$. Any finite $T$-transitive set~$F$ is not necessarily extendible into an infinite solution. Indeed, maybe there exist some~$a, b \in F$ such that $T(a, b)$ holds, but $T(b, x)$ and $T(x, a)$ both hold for cofinitely many~$x$. We shall therefore work with Mathias conditions $(\sigma, X)$ where $\sigma$ is a $T$-transitive finite set, with some extra structure which will guarantee that $\sigma$ is extendible into an infinite solution. This yields the following definition (due to Patey~\cite[Definition 5.7]{patey2015degrees}).

\begin{definition}[\cite{patey2015degrees}]
An \emph{$\EM$-condition} for~$T$ is a Mathias condition $(\sigma, X)$ such that
\begin{enumerate}
    \item for all $y \in X$, $\sigma \cup \{ y \}$ is $T$-transitive
    \item $X$ is included in a minimal $T$-interval of~$\sigma$
\end{enumerate}
\end{definition}

Actually, the second property can be obtained from the first one by a simple application of the infinite pigeonhole principle. Indeed, there are only finitely many minimal $T$-intervals in~$\sigma$, and each element of~$X$ belongs to exactly one of them. The notion of condition extension is the usual Mathias extension.

To simplify notation, given two disjoint sets~$F$ and $E$, we write $F \to_T E$ if for every~$a \in F$ and $b \in E$, $T(a, b)$ holds. One essential feature in the understanding of the computational content of a theorem is to understand the combinatorics necessary to extend a partial solution with an arbitrarily large number of elements in one block. In the case of the Erd\H{o}s-Moser theorem, the following lemma contains its core combinatorics.

\begin{lemma}[\cite{patey2015degrees}]\label[lemma]{lem:em-condition-combi}
Fix an $\EM$-condition $c = (\sigma, X)$ for a tournament~$T$, an infinite subset $Y \subseteq X$ and a finite $T$-transitive set~$\rho \subseteq X$ such that $\max \rho < \min Y$ and $[\rho \to_T Y \vee Y \to_T \rho]$. Then $(\sigma \cup \rho, Y)$ is a valid extension of~$c$.
\end{lemma}

Suppose one wants to design a good forcing question for deciding $\Sigma^0_1$ formulas with this notion of forcing. To simplify the situation, assume first that the tournament~$T$ is \emph{stable}, that if, for every~$a$, either $(\forall^\infty b) T(a, b)$ holds, or $(\forall^\infty b)T(b, a)$ holds. In other words, each element admits a limit behavior with respect to~$T$.
Let $f : \omega \to 2$ be the limit behavior of~$T$, that is, $f(a) = 0$ iff $\forall^\infty b T(a, b)$ and $f(a) = 1$ otherwise. The following naive definition does not satisfy the desired definitional properties:

\begin{definition}
Let $c = (\sigma,X) $ be an EM-condition, $n$ be an integer, and $e$ be a Turing index. Let $c \qvdash \Phi_e^G(n)\downarrow$ hold
    if there exists a finite $f$-homogeneous $T$-transitive set $\tau \subseteq X$ such that $\Phi_e^{\sigma \cup \tau}(n) \downarrow$.
\end{definition}

By \Cref{lem:em-condition-combi}, this is a valid forcing condition, in that if it holds, then there exists an extension forcing $\Phi_e^G(n)\downarrow$, and otherwise, there exists an extension forcing $\Phi_e^G(n)\uparrow$. From a definitional viewpoint, the previous relation is $\Sigma^0_1(X \oplus T \oplus f)$. However, the tournament~$T$ and its limit behavior $f$ are strongly non-computable, and may even compute the set~$B$ that we want to avoid. The solution to get rid of these parameters is to make an over-approximation:

\begin{definition}\label[definition]{def:forcing-question-em-level0}
Let $c = (\sigma,X) $ be an EM-condition, $n$ be an integer, and $e$ be a Turing index. Let  $c \qvdash \Phi_e^G(n)\downarrow$ hold
    if for every tournament~$R$ and every function $g : \NN \to 2$, there is a finite $g$-homogeneous $R$-transitive set~$\tau \subseteq X$ such that $\Phi_e^{\sigma \cup \tau}(n) \downarrow$.
\end{definition}

At first sight, an overapproximation yields a forcing question with even worse definitional properties since it contains a universal second-order quantification. However, thanks to compactness, the forcing question is actually $\Sigma^0_1(X)$, as it is equivalent to the following definition:

\begin{quote}
$c \qvdash \Phi_e^G(n)\downarrow$
    if there exists some threshold $t$ such that for every tournament~$R$ over~$\{0, \dots, t\}$ and every function $g : \{0, \dots, t\} \to 2$, there is a finite $g$-homogeneous $R$-transitive set~$\tau \subseteq X$ such that $\Phi_e^{\sigma \cup \tau}(n) \downarrow$.
\end{quote}

If the forcing question holds, then by letting $R = T$ and $g = f$, it is clear that there exists an extension forcing~$\Phi_e^G(n)\downarrow$. On the other hand, if the forcing question does not hold, then the witness of failure might be some tournaments~$R$ and some colorings~$f$ which are unrelated to~$T$ and $g$. This is where the combinatorics of Ramsey theory comes into play.

\begin{lemma}\label[lemma]{lem:questionf}
    Let $c = (\sigma,X) $ be an $\EM$-condition, $n$ be an integer and $e$ be a Turing index. 
    \begin{itemize}
        \item[(1)] If $c \qvdash \Phi_e^G(n)\downarrow$, then there exists $d \leq c$ such that $d \Vdash \Phi_e^G(n) \downarrow$.
        \item[(2)] Else, if $c \nqvdash \Phi_e^G(n) \downarrow$, then there exists $d \leq c$ such that $d \Vdash \Phi_e^G(n) \uparrow$.
    \end{itemize}
\end{lemma}

\begin{proof}[Proof in the stable case]
We prove each point :
    \begin{itemize}
        \item[(1)] If $c \qvdash \Phi_e^G(n)\downarrow$, letting~$R = T$ and $g = f$, there exists a finite $f$-homogeneous $T$-transitive set $\tau \subseteq X$ such that $\Phi_e^{\sigma \cup \tau}(n) \downarrow$. By choice of~$f$, there exists some $t \in \omega$ such that $\tau \to_T X \setminus \{0, \dots, t\}$ or $X \setminus \{0, \dots, t\} \to_T \tau$. Thus, by 
        \Cref{lem:em-condition-combi}, the pair $d := (\sigma \cup \tau, X \setminus \{0, \dots, t\})$ is an EM-condition. Note that $ d \leq c$ and that $d \Vdash \Phi_e^G(n)\downarrow$.

        \item[(2)] If $c \nqvdash \Phi_e^G(n) \downarrow$, then there exists a coloring $h : \NN \to 2$ and a tournament~$R$ such that for all finite $h$-homogeneous and $R$-transitive set $\tau \subseteq X$, $\Phi_e^{\sigma \cup \tau}(n) \uparrow$. By the pigeonhole principle and the Erd\H{o}s-Moser theorem restricted to~$X$, there exists an infinite subset~$Y \subseteq X$ which is both~$h$-homogeneous and $R$-transitive. The pair $d := (\sigma,Y)$ is a valid EM-condition such that $d \Vdash \Phi_e^G(e) \uparrow$.
        \end{itemize}
\end{proof}

Whenever the tournament is not stable, the situation seems more complicated as there is no clear choice of~$f$. Surprisingly, the previous forcing question still holds, but with a more subtle proof in the first case. The idea is the following: in the first case, by compactness, the finite extension candidate is bounded by a threshold. One can then restrict the reservoir, so that every element below the threshold has a limit behavior with respect to the new reservoir, and then act like in the stable case.

\begin{proof}[Proof in the general case]
We only prove the first case, as the second case did not involve stability of the tournament.
\begin{itemize}
        \item[(1)] If $c \qvdash \Phi_e^G(n)\downarrow$, by compactness, there exists some threshold $t$ such that for every tournament~$R$ over~$\{0, \dots, t\}$ and every function $g : \{0, \dots, t\} \to 2$, there is a finite $g$-homogeneous $R$-transitive set~$\tau \subseteq X$ such that $\Phi_e^{\sigma \cup \tau}(n) \downarrow$. 

        For every element~$y \in X \setminus \{0, \dots, t\}$, one can associate a function $g_y : \{0, \dots, t\} \to 2$ defined by~$g_y(x) = 1$ iff $T(x, y) = 1$.
        Since there are $2^t$ many such functions, then the function $y \mapsto g_y$ is a finite coloring of the reservoir, so by the infinite pigeonhole principle, there exists an infinite subset~$Y \subseteq X \setminus \{0, \dots, t\}$ which is homogeneous for some color~$g : \{0, \dots, t\} \to 2$. In other words, for every $g$-homogeneous set~$\tau \subseteq \{0, \dots, t\}$, either $\tau \to_T Y$, or $Y \to_T \tau$.

        Letting $R = T$, there is a finite $g$-homogeneous $T$-transitive set ~$\tau \subseteq X \cap \{0, \dots, t\}$ such that $\Phi_e^{\sigma \cup \tau}(n) \downarrow$. Thus, by 
        \Cref{lem:em-condition-combi}, the pair $d := (\sigma \cup \tau, Y)$ is an EM-condition. Note that $d \leq c$ and that $d \Vdash \Phi_e^G(n)\downarrow$.
\end{itemize}
\end{proof}

Together with the general discussion of \Cref{sect:picture-forcing-question} about forcing questions,
this section constitutes a proof that $\EM$ admits strong cone avoidance.

\begin{remark}
The bottom line of this section is the following: The combinatorics of the Erd\H{o}s-Moser theorem involve the pigeonhole principle, in that in order to ensure the extendibility of a finite $T$-transitive set, one needs to ensure that it is homogeneous for the appropriate instance of $\RT^1_2$. This 2-coloring represents the limit behavior of the tournament. Whenever the tournament is not stable, the choice of the 2-coloring is not clear ahead of time, and the colorings must be universally quantified.

Last, note that the use of $\RT^1_2$ in the proof of $\EM$ is not overkill, in that given a 2-coloring $f : \NN \to 2$, one can define a tournament $T$ by $T(x, y)$ iff $[x < y \leftrightarrow f(x) = f(y)]$. Then any infinite transitive subtournament is, up to finite changes, $f$-homogeneous.
\end{remark}

\subsection{Iterated jump control of EM forcing}\label[section]{sect:iterated-jump-picture}

In computability-theoretic forcing, one usually forces a $\Sigma^0_1/\Pi^0_1$ property in a strong sense: if $c \Vdash \varphi(G)$ for $\varphi \in \{\Sigma^0_1, \Pi^0_1\}$, then $\varphi(G_\F)$ actually holds for every filter~$\F$ containing~$c$. The situation becomes significantly more complicated when considering $\Sigma^0_2/\Pi^0_2$ formulas. 

A $\Pi^0_2$ formula $(\forall x)( \exists y) \varphi(G, x, y)$ can be considered as a countable collection of $\Sigma^0_1$ formulas $\{ (\exists y) \varphi(G, n, y) : n \in \NN \}$. Such a formula cannot usually be forced in a strong sense. The relation $c \Vdash (\forall x )(\exists y) \varphi(G, x, y)$ holds iff for every $x \in \NN$, the set of conditions forcing $(\exists y) \varphi(G, x, y)$ is dense below~$c$. This way, every sufficiently generic filter containing $c$ will also contain a condition forcing~$(\exists y )\varphi(G, x, y)$ for each~$x \in \NN$, thus the property $(\forall x )(\exists y) \varphi(G_\F, x, y)$ holds for every sufficiently generic filter containing~$c$.

Stating the density of a collection of conditions can be a definitionally complex statement, depending on the complexity of the notion of forcing. In some simple cases, such as Cohen forcing, the forcing relation for $\Pi^0_2(G)$ formulas is $\Pi^0_2$, yielding a good forcing question.

Variants of Mathias forcing do not behave that well. Indeed, the statement of density requires universal and existential quantification on the conditions, hence on the reservoirs, which are second-order objects. Actually, the approach of Mathias forcing provably fails: 

\begin{lemma}[Folklore]
The set~$\emptyset''$ is $\Pi^0_2(G_\F)$ for every sufficiently generic filter~$\F$ for Mathias forcing with computable reservoirs.
\end{lemma}
\begin{proof}
By Martin's domination theorem, a set is of high degree iff it computes a function eventually dominating every total computable function. Given a computable function~$f$ and a computable Mathias condition $(\sigma, X)$, there exists a computable Mathias extension $(\sigma, Y)$ such that the principal function of~$Y$ (the function which to $n$ associates the $n$th element of~$Y$) dominates~$f$. Thus, for every sufficiently generic filter~$\F$, the principal function of~$G_\F$ will eventually dominate every total computable function, hence be of high degree.
\end{proof}

%It is well-known that every sufficiently generic filter for computable Mathias forcing, that is, Mathias forcing with computable reservoirs, is of high degree.  Indeed, given a computable function~$f$ and a computable Mathias condition $(\sigma, X)$, there exists a computable Mathias extension $(\sigma, Y)$ such that the principal function of~$Y$ (the function which to $n$ associates the $n$th element of~$Y$) dominates~$f$. By Martin's  It follows that $\emptyset''$ is not $\Pi^0_2$, but is $\Pi^0_2(G_\F)$ for every sufficiently generic filter for computable Mathias forcing. 

The same argument holds for computable EM forcing.
Intuitively, the reason of failure of Mathias forcing and its variants, is because of the sparsity of its reservoirs. One way to circumvent the problem is to restrict the class possible reservoirs with a third-order object, which will play the role of a \qt{reservoir of reservoirs}: this meta-reservoir is a class of infinite sets~$\L$ which cannot contain arbitrarily sparse objects. Our goal is to work with EM conditions $(\sigma, X)$ such that $X \in \L$. One however requires the class~$\L$ to be closed under certain operations which are needed for using the combinatorics of~\Cref{sect:combinatorics-em}. Analyzing the operations made over the reservoirs, they are of three kinds:
\begin{enumerate}
    \item truncation of a reservoir from a finite number of elements (case 1 of the forcing question)
    \item splitting of a reservoir based on an instance of~$\RT^1_2$ (case 2 of the forcing question)
    \item choice of an $R$-transitive subtournament for an instance~$R$ of~$\EM$ (case 2 of the forcing question)
\end{enumerate}
Assuming~$\L$ contains only infinite sets, the truncation operation is a consequence of finitely many applications of~$\RT^1_2$. Unfortunately, contrary to the pigeonhole principle, the classes which are closed under applications of~$\EM$ do not have nice combinatorial properties. Therefore, in Case 2 of \Cref{lem:questionf}, instead of applying the Erd\H{o}s-Moser theorem restricted to~$X$ to obtain an infinite $R$-transitive subtournament~$Y \subseteq X$, we shall simply add~$R$ to the list of tournaments we commit to be transitive for. The benefit of it is that the only remaining operation done on our reservoirs is the application of~$\RT^1_2$. The counterpart of postponing our application of~$\EM$ is that our forcing conditions will now be made of triples $(\vec{R}, \sigma, X)$, where~$\vec{R}$ is a finite sequence of tournaments, and such that $(\sigma, X)$ is an EM condition for every~$R \in \vec{R}$. This list~$\vec{R}$ can grow with condition extension. 

The reservoir~$X$ will therefore belong to a class~$\L$ which is closed under applications of~$\RT^1_2$. This notion of closure is called \emph{partition regularity}. We shall introduce this concept and its main properties in \Cref{sect:largeness-pr}.
The restriction of the reservoirs to those which belong to a partition regular class dramatically decreases the definitional complexity of the forcing question, as instead of asking whether for every infinite set~$Y \subseteq X$, there is an infinite set~$Z \subseteq Y$ satisfying some property, one can ask whether the collection of all $Z$ satisfying the property is partition regular. Based on the complexity of the property, the question will not be definitionally too complex.

\section{Partition regularity and largeness}\label[section]{sect:largeness-pr}

The notion of partition regularity comes from combinatorics and is widely used in Ramsey theory. It therefore naturally occurred in the computability-theoretic analysis of combinatorial theorems. 

\begin{definition}
A class $\L \subseteq 2^{\omega}$ is \textit{partition regular} if :
\begin{itemize}
    \item $\L$ is non-empty,
    \item for all $X \in \L$, if $ X \subseteq Y$, then $Y \in \L$,
    \item for every integer $k$, for every $X \in \L$, for every $k$-cover $Y_1, Y_2, \dots Y_k$ of $X$, there exists $i \leq k$ such that $Y_i \in \L.$
\end{itemize}
\end{definition}

Dorais~\cite{dorais2012variant} was the first to use variants of Mathias forcing in which the reservoirs belong to partition regular classes, to produce generic sets of non-high degree. Since then, Monin and Patey~\cite{monin2019pigeons,monin2021srt22,monin2021weakness,monin2022partition} successfully used this variant to control iterated jump of solutions to the infinite pigeonhole principle. Monin and Patey~\cite[Section 2]{monin2021weakness} contains all the computability-theoretic analysis of partition regularity used in this article.
In this section, we therefore simply state the relevant definitions and theorems for the sake of completeness.

Partition regularity enjoys nice closure properties, but is not a notion of largeness per se, in that a superclass of a partition regular class is not necessarily partition regular itself. Throughout the article, given a property $\varphi(X)$, we will ask whether the class $\{ X : \varphi(X) \}$ is large, in the sense that it contains a partition regular subclass. This yields the following definition, which is often more convenient to manipulate than partition regularity.

\begin{definition}
A class $\L \subseteq 2^{\omega}$ is \textit{large} if :
\begin{itemize}
    \item for all $X \in \L$, if $ X \subseteq Y$, then $Y \in \L$,
    \item for every integer $k$, for every $k$-cover $Y_1, Y_2, \dots Y_k$ of $\omega$, there exists $i \leq k$ such that $Y  \in \L.$
\end{itemize}
\end{definition}

The large classes are exactly those which contain a partition regular subclass. To avoid degenerate behaviors, in this article, we shall restrict ourselves to large classes which contain only infinite sets. 

\begin{definition}
    A large class is \textit{non-trivial} if it contains only infinite sets.
\end{definition}

In particular, if a partition regular class~$\P$ is non-trivial, then it is closed under finite changes. Indeed, if~$X \in \P$ and $Y =^* X$ as witnessed by a finite set~$F$, then $Y, F$ form a 2-cover of~$X$, so either~$Y$ or $F$ must belong to~$\P$, and by non-triviality, $F \not \in \P$.

One of the core properties of large classes is the following lemma, which plays an essential role in the computability-theoretic analysis of large classes. For example, by contraposition, if an $F_\sigma$ class is not large, then it is included into a non-large open class.

\begin{lemma}[\cite{monin2021weakness}]\label[lemma]{lem:intersection-still-large}
Let $(\P_n)_{n \in \omega}$ be a decreasing sequence of large classes. Their intersection $\bigcap_{n \in \omega} \P_n$ is again large.
\end{lemma}

The above lemma also holds if one replaces largeness by partition regularity. Moreover, a union of partition regular classes is still partition regular. Therefore, every large class contains a largest (for inclusion) partition regular subclass, which justifies the following definition.

\begin{definition}
    For every large class $\P$, let $\L(\P)$ denote the largest partition regular subclass of $\P$.
\end{definition}

By the infinite pigeonhole principle, the class of all infinite sets is partition regular. This naturally generalizes to a whole family of partition regular classes:

\begin{definition}
    For every set $X \subseteq \omega$, let $\L_X := \{ E \subseteq \omega : |E \cap X|=\infty \}$.
\end{definition}

For every infinite set~$X$, the class~$\L_X$ is partition regular. This class plays an essential role in ensuring that a set belongs to all partition regular subclasses. Indeed, if $\P$ is a partition regular subclass of~$\L_X$, then $X \in \P$, as otherwise, since $X, \overline{X}$ forms a 2-cover of~$\omega$, we would have~$\overline{X} \in \P \subseteq \L_X$, contradiction.

\subsection{$\Pi^0_2$ large classes}

By Monin and Patey~\cite[Proposition 2.15]{monin2022partition}, there are no non-trivial $\mathbf{\Sigma^0_2}$ large classes. The first interesting example of such classes are~$\Pi^0_2$.
Along this article, we will only be interested in $F_\sigma$ classes, and more precisely intersections of~$\Sigma^0_1(\M)$ classes, for some Scott set~$\M$ (recall that a formula is $\Sigma^0_1(\M)$ if it is $\Sigma^0_1(P)$ for some~$P \in \M$). Fix a Scott set $\M$ encoded by a set $M \subseteq \omega$, i.e., $ M = \bigoplus_{n \in \omega} X_n$ and $\M = \{ X_n : n \in \omega \}$. One can code such classes by sets $C \subseteq \omega^2$ as follows:

\begin{definition}
Fix a set $P$. For every~$e \in \omega$, let $\U^P_e = \{ Z \in 2^\omega : \exists \sigma \in W^P_e : \sigma \subseteq Z \}$.
For every $C \subseteq \omega^2$, let $\U^{\M}_C = \bigcap_{(e,i) \in C} \U^{X_i}_e$.
\end{definition}

The following lemma is a core lemma in the analysis of the definitional complexity of the statement \qt{$\U^{\M}_C$ is large}, thanks to \Cref{lem:intersection-still-large}.

\begin{lemma}[\cite{monin2021weakness}]\label[lemma]{largenesssentencecomp}
Let $\mathcal{A}$ be a $\Sigma_1^0$ class. The sentence \qt{$ \mathcal{A}$ is large} is $\Pi_2^0$.
\end{lemma}

By \Cref{lem:intersection-still-large}, a class $\U^{\M}_C$ is large iff $\U^{\M}_F$ is large for every finite set~$F \subseteq C$. The class $\U^{\M}_F$ is $\Sigma^0_1(\M)$ uniformly in~$F$, hence by a relativization of \Cref{largenesssentencecomp}, the statement \qt{$\U^{\M}_F$ is large} is $\Pi^0_2(\M)$ uniformly in~$F$. The overall statement \qt{$\U^{\M}_C$ is large} is therefore $\Pi^0_1(C \oplus M')$, where $M$ is the set coding $\M$.

The following lemma shows that instead of working with large classes of the form $\U^\M_C$,
one can work with partition regular classes without extra definitional complexity.

\begin{lemma}[\cite{monin2021weakness}]\label[lemma]{lem:compute-luc}
    For every set $C \subseteq \omega^2$, there exists $D \leq_T C$ such that $\U^\M_D = \L(\U^\M_C)$.
\end{lemma}

\subsection{$\M$-minimal and $\M$-cohesive classes}\label[section]{sect:minimal-cohesive-classes}

When two classes $\P$ and $\Q$ are large, their intersection $\P \cap \Q$ is not necessarily large. For example, letting~$X$ be a bi-infinite set, both the classes $\L_X$ and $\L_{\overline{X}}$ are large, but their intersection is not, as witnessed by the 2-cover $X, \overline{X}$ of~$\omega$.
During the forcing construction, one will consider multiple properties to be forced, and therefore will need to ensure that not only the corresponding classes are large, but so are their intersection. One natural approach consists in creating a large class which will be minimal for inclusion, in the following sense:

\begin{definition}
      A class $\A$ is \emph{$\M$-minimal} if for every $X \in \M$ and $e \in \omega$, either $\A \subseteq \U_e^X$ or $\A \cap \U_e^X$ is not large.
\end{definition}

Then, in order to decide whether two $\Sigma^0_1(\M)$ properties $\P$ and $\Q$ are large, one can ask independently whether $\A \cap \P$ and $\A \cap \Q$ are large. If both are, then by $\M$-minimality of~$\A$, $\A \subseteq \P$ and $\A \subseteq \Q$, hence $\P \cap \Q$ is large as well.
Eventhough the notion of minimality was defined for largeness, partition regularity comes for free for an $\M$-minimal large class.

\begin{lemma}[\cite{monin2021weakness}]\label[lemma]{lem:minimal-partition-regular}
    Every $\M$-minimal large class $\U_C^\M$ is partition regular.
\end{lemma}

There exists $\M$-minimal large classes of the form $\U^\M_C$. However, the index set~$C$ is computationally too complex, as it is only~$M''$-computable. Indeed, in order to create the set~$C$ by finite approximations $C_0 \subseteq C_1 \subseteq \dots$, one needs to successively ask whether $\U^\M_{C_s} \cap \U_e^X$ is large, which is a $\Pi^0_2(\M)$ question. Thankfully, one can consider a weaker notion with better computational properties, which still satisfies the compatibility requirements.

\begin{definition}
    A class $\A$ is \emph{$\M$-cohesive} if for every $X \in \M$, either $\A \subseteq \L_X$ or $\A \subseteq \L_{\overline{X}}$.
\end{definition}

Every $\M$-minimal class is $\M$-cohesive. Moreover, one can compute the index set~$C$ of an $\M$-cohesive class~$\U^\M_C$ in any PA degree over~$M'$. Indeed, instead of deciding whether $\U^\M_{C_s} \cap \L_X$ is large or not, one needs to pick a true statement among \qt{$\U^\M_{C_s} \cap \L_X$ is large} and \qt{$\U^\M_{C_s} \cap \L_{\overline{X}}$ is large}. Choosing a true $\Pi^0_2(M)$ sentence among two such sentences, known that one of them is true, can be computed by any PA degree over~$M'$ (see Cholak, Jockusch and Slaman~\cite[Lemma 4.2]{cholak_jockusch_slaman_2001} for a proof).

\begin{lemma}[\cite{monin2021weakness}]\label[lemma]{lem:cohesive-compatibility}
    Let $\U_C^{\M}$ be an $\M$-cohesive class. Let $\U_D^{\M}$ and  $\U_E^\M$ be such that $\U_C^\M \cap \U_D^\M$ and  $\U_C^\M \cap \U_E^\M$ are both large. Then so is $\U_C^\M \cap \U_D^\M \cap  \U_E^\M$.
\end{lemma}

In general, given a large class~$\U^\M_C$, the index set~$C$ can be completed into an index set~$D \supseteq C$ in multiple ways to form an $\M$-minimal large subclass~$\U^\M_D$, depending on the order in which the questions are asked. However, whenever $\U_C^{\M}$ is $\M$-cohesive, then by \Cref{lem:cohesive-compatibility}, the order of the questions does not matter, thus it contains a unique $\M$-minimal subclass, which can be explicitly defined as follows: 

\begin{lemma}
    Given an \MCO large class $\UCM$, the collection of sets 
    $$ \langle \UCM \rangle := \bigcap_{e \in \omega, X \in \M} \{ \U_e^X : \UCM \cap \U_e^X \text{ is a large } \} $$
    is an $\M$-minimal large class contained in $\UCM$.
\end{lemma}

\subsection{The $(\U_{C_n}^{\M_n})_{n \in \omega}$ sequence}\label[section]{sect:uc-sequence}

Monin and Patey~\cite[Section 2.5]{monin2021weakness} defined an infinite hierarchy of Scott sets together with a decreasing sequence of minimal classes for these Scott sets, playing a central role in the definition of the notion of forcing. The $n$th level of this hierarchy is responsible for having a good $(n+1)$st jump control.

The following first proposition is an easy consequence of the uniform low basis theorem for~$\Pi^0_1$ classes, due to Lawton (see Hirschfeldt and al.~\cite[Theorem 4.1]{hirschfeldt2008strength}):

\begin{proposition}[\cite{monin2021weakness}]\label[proposition]{mn-seq}
    There exists a sequence of sets \seq{M} such that :
    \begin{itemize}
        \item $M_n$ codes for a countable Scott set $\M_n$,
        \item $\emptyset^{(n)}$ is uniformly coded by an element of $\M_n$,
        \item Each $M_n'$ is uniformly computable in $\emptyset^{(n+1)}$.
    \end{itemize}
\end{proposition}

Then, the next proposition follows from our remark on the complexity of the construction 
of an $\M$-cohesive large class. Indeed, since~$\M_{n+1}$ contains~$M'_n$ and $\M_{n+1}$ is a Scott set, then it contains a PA degre over~$M'_n$. Moreover, this construction is uniform.

\begin{proposition}[\cite{monin2021weakness}]\label[proposition]{cn-seq}
    There exists a sequence of sets \seq{C} such that :
    \begin{itemize}
        \item $\UCNMN$ is an $\M_n$-cohesive large class,
        \item $\U_{C_{n+1}}^{\M_{n+1}} \subseteq \langle \UCNMN \rangle$,
        \item Each $C_n$ is coded by an element of $\M_{n+1}$ uniformly in $n$ and $M_{n+1}$.
    \end{itemize}
\end{proposition}

\section{Forcing framework}\label[section]{sect:forcing-framework}

We now develop the general framework for iterated jump control of solutions to the Erd\H{o}s-Moser theorem through forcing. This framework will be applied in sections~\ref{sect:strong-avoidance-arith}, \ref{sect:layerwise-avoidance} and~\ref{sect:effective-constructions} to prove our three main theorems.

\subsection{Forcing conditions}

In order to obtain a layerwise version of strong cone avoidance of~$\EM$ for $\Sigma^0_n$ operators, our notion of forcing will be parameterized by a partition regular class~$\P$. 
Assuming $\P \subseteq \langle \UCNMN \rangle$, this notion of forcing will have a good $(n+1)$st jump control. All along \Cref{sect:forcing-framework}, one should think of~$\P$ as the partition regular class~$\bigcap_{n \in \omega} \langle \UCNMN \rangle$.

\begin{definition}\label[definition]{def:condition}
    Given a partition regular class~$\P \subseteq 2^\omega$, let $\PP_\P$ denote the set of all 3-tuples
     $(\vec{R},\sigma,X)$ such that
    \begin{enumerate}
        \item \vr~is a finite sequence of tournaments,
        \item $X \cap \{0, \dots, |\sigma| \} = \emptyset$, 
        \item $X \in \P$.
        \item for all $y \in X$, $\sigma \cup \{y\}$ is \vr-transitive.
        \item $X$ is included in a minimal $\vec{R}$-interval of~$\sigma$
    \end{enumerate}
\end{definition}

In other words, $\PP_\P$ is the set of all 3-tuples $(\vec{R}, \sigma, X)$ such that $(\sigma, X)$ forms an EM condition for each tournament~$R \in \vec{R}$, and such that $X \in \P$. Note that no effectiveness restriction is given on the reservoir~$X$. Given a tournament~$T$, in order to produce an infinite $T$-transitive subset, one will work with sufficiently generic filters containing the condition $(T, \emptyset, \omega)$.
From now on, fix a partition regular class~$\P$.
%To simplify notation, from now on, we will fix a partition regular class~$\P$.
%$n \in \omega$. Let $\PP_n$ denote the set of all 3-tuples

%\ludovic{TODO, give intuitions as EM conditions. Note that no effectiveness restriction is given on the reservoir~$X$.}

\begin{definition}
    We define the partial order over $\PP_\P$ as following : we say that $(\tau,Y, \vec{S}) \leq (\sigma,X, \vec{R})$ if $\sigma \preceq \tau, Y \subseteq X$, $\tau \setminus \sigma \subseteq X$, and $\vec{R} \subseteq \vec{S}$.
\end{definition}

Here again, the extension relation $(\tau,Y, \vec{S}) \leq (\sigma,X, \vec{R})$ is the usual  Mathias extension $(\tau, Y) \leq (\sigma, X)$, but in addition, one commits to be transitive for more and more tournaments simultaneously.
Given a collection $\F \subseteq \PP_\P$, we let $G_{\F} := \bigcup_{(\vec{R},\sigma,X) \in \F} \sigma$.

\begin{lemma}\label[lemma]{lem:gf-transitive}
    Let $\F$ be a $\PP_\P$-filter. For all $c := ( \vec{R},\sigma,X)$, the set $G_\F$ is $\vec{R}$-transitive.
\end{lemma}

\begin{proof}
    Suppose otherwise. Then, there exists $x < y < z \in G_\F$ a $\vec{R}$-cycle. By definition of $G_\F$, there exists $d :=( \vec{R}\vec{S}, \tau, Y) \in \F$, such that $ d \leq c$ and $\{ x,y,z \} \subseteq \tau$. Since $\tau$ is $\vec{R}\vec{S}$-transitive, and thus $\vec{R}$-transitive, $(x,y,z)$ cannot be a $\vec{R}$-cycle.
\end{proof}

\subsection{Forcing question}

In this section, we design a forcing question as explained in \Cref{sect:picture-forcing-question}. The general idea is the following: given a $\Sigma^0_1$ formula $(\exists x )\varphi(G, x)$ and a condition $(\vec{R}, \sigma, X)$, one would like to ask whether there exists some~$x \in \omega$ and a finite set $\tau \subseteq X$ satisfying some good combinatorial properties, such that $\varphi(\sigma \cup \tau, x)$ holds. This naive question is $\Sigma^0_1(\vec{R} \oplus X)$. As explained in \Cref{sect:combinatorics-em}, one can get rid of the parameter~$\vec{R}$ by universally quantifying over all $m$-tuples of tournaments, where~$m = |\vec{R}|$. By compactness, the question becomes~$\Sigma^0_1(X)$, which is not enough, since~$X$ can be computationally very complex. The same overapproximation trick cannot be applied for~$X$, since the class of all infinite sets is closed, but not compact. One must therefore use a second trick: consider the class~$\L$ of all reservoirs~$Y$ such that this property holds, that is, of all $Y$ such that there exists some~$x \in \omega$ and a finite set $\tau \subseteq X$ satisfying some good combinatorial properties, such that $\varphi(\sigma \cup \tau, x)$ holds. Then, ask whether~$\L \cap \U^{\M_0}_{C_0}$ is large. If it does, then by $\M_0$-minimality of~$\P$, $\P \subseteq \L$, and since~$X \in \P$, the property holds for~$X$ in particular.

Before giving the actual definition of the forcing question, let us introduce a very convenient piece of notation. As explained, given a condition $(\vec{R}, \sigma, X)$, since no effectiveness constraint is imposed on~$\vec{R}$, one will often resort to an over-approximation of~$\vec{R}$. This over-approximation~$\C$ has two essential properties: (1) it must contain~$\vec{R}$, and (2) it must be $X$-effectively compact. One can exploit the $\Pi^0_1$ constraints in the definition of a condition to obtain a finer over-approximation: $X$ is included in a minimal interval of~$\sigma$ and $\sigma \cup \{y\}$ is $\vec{R}$-transitive for every~$y \in X$. This yields the following definition:

\begin{definition}
Let~$(\sigma, X)$ be a Mathias condition.
For every~$m \in \omega$, $\C_m(\sigma, X)$ is the class of all $m$-tuples $\vec{R}$ of tournaments such that for all $y \in X$, $\sigma \cup \{y\}$ is \vr-transitive, and such that $X$ is included in a minimal $\vec{R}$-interval of $\sigma$. 
\end{definition}

In particular, for every condition $(\vec{R}, \sigma, X)$, letting~$m = |\vec{R}|$, $\vec{R} \in \C_m(\sigma, X)$. Thus, the class~$\C_m(\sigma, X)$ is an over-approximation of~$\vec{R}$. On the other direction, if $(\sigma, X)$ is a Mathias condition such that $X \in \P$, then for all $m \in \omega$ and for all $\vec{R} \in \C_m(\sigma,X)$, the 3-tuple $(\vec{R}, \sigma, X)$ is a condition in $\PP_\P$. The following lemma shows that the over-approximation~$\C_m(\sigma, X)$ is $X$-effectively compact, as desired.

\begin{lemma}\label[lemma]{lem:approx-pi01}
For every Mathias condition $(\sigma, X)$ and every~$m \in \omega$,
$\C_m(\sigma, X)$ is $\Pi^0_1(X)$.
\end{lemma}
\begin{proof}
Immediate. Both properties are $\Pi^0_1$ formulas in~$X$.
\end{proof}

% \begin{remark}
%     If moreover $X \in \P$, then for all $m \in \omega$ and for all $\vec{R} \in \C_m(\sigma,X)$, the 3-tuple $(\vec{R}, \sigma, X)$ is condition in $\PP_\P$.
% \end{remark}

The following combinatorial lemma is essentially a reformulation of \Cref{lem:em-condition-combi}. If one furthermore assumes that~$Y \in \P$, then $(\vec{R}, \sigma \cup \tau, Y)$ is a valid $\PP_\P$-condition.

\begin{lemma}\label[lemma]{lem:extension}
    Let $(\vec{R},\sigma,X)$ be a condition, $\tau \subseteq X$ be a finite $\vec{R}$-transitive set and $Y \subseteq X$ be an infinite set such that for every~$R \in \vec{R}$, $\tau \to_R Y$ or $Y \to_R \tau$. Then, $\vec{R} \in \C_{|\vec{R}|}(\sigma \cup \tau, Y)$.
\end{lemma}

\begin{proof}

Since~$(\vec{R},\sigma, X)$ is a condition, $X$ is included in a minimal interval of~$\sigma$ for all $R \in \vec{R}$. Furthermore, since~$\tau \cup Y \subseteq X$, and $\tau \to_R Y$ or $Y \to_R \tau$, then $Y$ is included in a minimal interval of $\sigma \cup \tau$ for all $R \in \vec{R}$.

Let $y \in Y$.
Suppose for the contradiction that there exists a 3-cycle $x < y < z$ in $\sigma \cup \tau \cup \{y\}$. 
\begin{itemize}
    \item If $x, y \in \sigma$ then since~$(\vec{R}, \sigma, X)$ is a condition, $\sigma \cup \{z\}$ is $R$-transitive, so $\{x,y,z\}$ is not a 3-cycle.
    \item If $x \in \sigma$ and $y, z \in \tau \cup \{a\}$, then since $X$ is included in a minimal interval of~$\sigma$, $x \to_R \{y, z\}$ or $\{y, z\} \to_R x$, hence $\{x,y,z\}$ is not a 3-cycle.
    \item If $x, y, z \in \tau$, then $\{x,y,z\}$ is not a 3-cycle by $R$-transitivity of $\tau$
    \item If $x, y \in \tau$ and $z = a$, then since $\tau \to_R Y$ or $Y \to_R \tau$, then $\{x,y,z\}$ is not a 3-cycle.
\end{itemize}
\end{proof}

We are now ready to define the forcing question. Since $\vec{R}$ is only accessed through an over-approximation, and $X$ through largeness, the forcing question is parameterized only by the initial segment $\sigma$ and the number $m$ of tournaments, rather than by the condition~$(\vec{R}, \sigma, X)$.

\begin{definition}
Let  $c := ( \vec{R}, \sigma,X) \in \PP_\P$ be a condition and $m \geq |\vec{R}|$. Consider $(\exists x) \psi_e(G,x)$ a $\Sigma_1^0$ formula.  We define the $\qvdash$ relation as follows:

  $$ \sigma \qvdash_m (\exists x)~\psi_e(G,x) $$
    holds if the class of all $Y \in \U_{C_0}^{\M_0}$ such that
    for every~$m$-tuple of 2-colorings $\vec{h} = h_0, \dots, h_{m-1} \in 2^Y$ and every $\vec{S} \in \C_{m}(\sigma,Y)$, there is a finite $\tau \subseteq Y \setminus \{0, \dots, |\sigma|\}$ which is $\vec{S}$-transitive and $\vec{h}$-homogeneous, and some~$x \in \omega$ such that $ \psi_e(\sigma \cup \tau,x)$
    is large.

    Inductively, for $n \geq 1$, consider $(\exists x) \psi_e(G,x)$ a $\Sigma_{n+1}^0$ formula. We define the $\qvdash$ relation as follows:

 $$ \sigma \qvdash_m (\exists x)~\psi_e(G,x)$$
   holds if the class of all $Y \in \U_{C_n}^{\M_n}$ such that
    for every~$m$-tuple of 2-colorings $\vec{h} = h_0, \dots, h_{m-1} \in 2^Y$ and every $\vec{S} \in \C_{m}(\sigma,Y)$, there is a finite $\tau \subseteq Y \setminus \{0, \dots, |\sigma|\}$ which is $\vec{S}$-transitive and $\vec{h}$-homogeneous, some $x \in \omega$ and $\ell \geq m$ such that $\sigma \cup \tau \nqvdash_\ell \neg \psi_e(G, x)$
    is large.
\end{definition}

The over-approximation of~$\vec{R}$ comes at no extra cost from a definitional complexity viewpoint. On the other hand, over-approximating the reservoir~$X$ by a large class yields a $\Pi_1^0(\M_{n+1})$ forcing question for~$\Sigma^0_{n+1}$ formulas, which is sufficient for arithmetic reductions, but not a layerwise cone avoidance.

\begin{lemma}\label[lemma]{lem:question-below-complexity}
    Let $n \in \omega$ and $c := ( \vec{R}, \sigma,X) \in \PP_\P$. Consider $(\exists x) \psi_e(G,x)$ a $\Sigma_{n+1}^0$ formula. The formula $(\sigma \qvdash_m (\exists x) \psi_e(G,x))$ is $\Pi_1^0(\M_{n+1})$, for all $m \in \omega$.
\end{lemma}

\begin{proof}

    We prove this result inductively over~$n$.
    The sentence is of the form \qt{$\L \cap \U_{C_n}^{\M_n}$ is large}, where
    $\L$ is the class of all $Y$ such that
    for every~$m$-tuple of 2-colorings $\vec{h} = h_0, \dots, h_{m-1} \in 2^Y$ and every $\vec{S} \in \C_{m}(\sigma,Y)$, there is a finite $\tau \subseteq Y \setminus \{0, \dots, |\sigma|\}$ which is $\vec{S}$-transitive and $\vec{h}$-homogeneous, and some~$x \in \omega$ such that 
    \begin{itemize}
        \item for $n = 0$, $\psi_e(\sigma \cup \tau,x)$.
        \item for $n > 0$, there is some $\ell \geq m$ such that $\sigma \cup \tau \nqvdash_\ell \neg \psi_e(G, x)$.
    \end{itemize}

    By a compactness argument, $\L$ is also the class of all $Y$ such that there exists $t \in \omega$ such that for every $m$-tuple of 2-colorings $\vec{h} = h_0, \dots, h_{m-1} \in 2^t$ and every $m$-tuple  $\vec{S}$ of tournaments over $\{0, \dots t\}$ such that for all $y \in \{0,\dots,t \} \cap Y, \sigma \cup \{y\}$ is $\vec{S}$-transitive, there is a finite $\tau \subseteq \{0, \dots ,t \} \cap Y \setminus \{0, \dots, |\sigma|\}$ which is $\vec{S}$-transitive and $\vec{h}$-homogeneous, and some~$x \in \omega$ such that
    \begin{itemize}
        \item for $n = 0$, $\psi_e(\sigma \cup \tau,x)$.
        \item for $n > 0$, there is some $\ell \geq m$ such that $\sigma \cup \tau \nqvdash_\ell \neg \psi_e(G, x)$.
    \end{itemize}
    The first item is $\Sigma^0_1$, and by induction hypothesis, the second item $\sigma \cup \tau \nqvdash \neg \psi_e(G,x)$ is $\Sigma_1^0(\M_n)$, hence $\L$ is a  $\Sigma^0_1(\M_n)$ class.

    By \Cref{lem:intersection-still-large}, the class $\L \cap \U_{C_n}^{\M_n}$ is large if and only if for all finite set $F \subseteq C_n$, $\L \cap \U_F^{\M_n}$ is also large. Since $\L \cap \U_F^{\M_n}$ is $\Sigma^0_1(M_n)$ uniformly in~$F$, then by a relativized \Cref{largenesssentencecomp}, the sentence \qt{$\L \cap \U_F^{\M_n}$ is large} is $\Pi_2^0(M_n)$ uniformly in~$F$, hence a $\Pi_1^0(M_n')$ sentence uniformly in $F$, and thus, by \Cref{mn-seq}, a $\Pi_1^0(\emptyset^{(n+1)})$ sentence. This makes $\L \cap \U_{C_n}^{\M_n}$ largeness a $\Pi_1^0(C_n \oplus \emptyset^{(n+1)})$ sentence, Moreover, by \Cref{cn-seq} and \Cref{mn-seq},  $(C_n \oplus \emptyset^{(n+1)}) \in \M_{n+1}$, hence, the sentence \qt{$\L \cap \U_{C_n}^{\M_n}$ is large} is a $\Pi_1^0(\M_{n+1})$ sentence.

\end{proof}

We now define the forcing relation for arithmetic formulas. The base cases for $\Sigma^0_1$ and $\Pi^0_1$ formulas, as well as the $\Sigma^0_{n+1}$ case, are quite straightforward. The interesting case is for~$\Pi^0_{n+1}$ formulas $(\forall x)\neg \psi_e(G, x)$: it asserts that for every~$x \in \omega$ and extension~$d = (\vec{R}\vec{S}, \sigma \cup \tau, Y)$ of~$c$, the forcing question $\sigma \cup \tau \qvdash_{|\vec{R}\vec{S}|} \neg \psi_e(G, x)$ will hold. Assuming that the forcing question meets its specifications, that is, if the forcing question holds for a formula, then there exists an extension forcing this formula, then the forcing relation for~$\Pi^0_{n+1}$ formulas is a density statement: it asserts that for every~$x \in \omega$, the set of conditions forcing~$\neg \psi_e(G, x)$ is dense below~$c$. Thus, for every sufficiently generic filter~$\F$ containing~$c$ and every~$x \in \omega$, there will be a condition~$d_x \in \F$ forcing~$\neg \psi_e(G, x)$.

\begin{definition}
    Let $c := (\vec{R},\sigma,X) \in \PP_\P$. Consider $(\exists x) \psi_e(G,x)$ a $\Sigma_{1}^0$ formula. We define the $\Vdash$ relation as follows:
    \begin{itemize}
        \item $c \Vdash (\exists x) \psi_e(G,x)$ if $(\exists x) \psi_e(\sigma,x)$.
         \item $ c \Vdash (\forall x) \neg \psi_e(G,x)$ if $(\forall \tau \subseteq X)(\forall x)( \tau$ is $\vec{R}$-transitive $ \implies \neg \psi_e(\sigma \cup \tau,x))$.
    \end{itemize}

    \medskip
    Then, inductively, for $n \geq 1$, let $(\exists x) \psi_e(G,x)$ be a $\Sigma_{n+1}^0$ formula. Then,
    \begin{itemize}
        \item $ c \Vdash (\exists x) \psi_e(G,x)$ if $(\exists x)(c \Vdash \psi_e(G,x))$.
        \item  $ c \Vdash (\forall x) \neg \psi_e(G,x)$ if $(\forall \tau \subseteq X)(\forall x)(\forall \ell \geq |\vec{R}| )(\tau$ is $\vec{R}$-transitive$ \implies \sigma \cup \tau \qvdash_\ell \neg \psi_e(G,x))$.
    \end{itemize}
    
\end{definition}

\begin{remark}\label[remark]{rem:monotonicity-sigma01}
Every $\Sigma^0_1(G)$ formula $\psi(G)$ can be expressed without loss of generality of the form $\Phi^G_e(0)\downarrow$. By the use property, the notion of Turing functional can be extended to finite length oracles, which induces an extension of the formula $\psi(G)$ to finite strings, such that $\psi(G)$ holds iff $\psi(G \uh_k)$ holds for some~$k \in \omega$. Moreover, the formula $\psi(\sigma)$ can be chosen so that if $\psi(\sigma)$ holds, then so does $\psi(\tau)$ for every~$\tau \succeq \sigma$.
Throughout this article, we will always assume that $\Sigma^0_1$ formulas are in this normal form.
\end{remark}

The following lemma shows that the forcing relation is stable under condition extension.

\begin{lemma}\label[lemma]{lem:forcing-relation-closure}
    Fix $n \geq 0$. Let $d, c \in \PP_\P$ be such that $d \leq c$, and let $(\exists x) \psi_e(G,x)$ be a $\Sigma_{n+1}^0$ formula. 
    \begin{itemize}
        \item[(1)] If $c \Vdash (\exists x) \psi_e(G,x)$ then $d \Vdash (\exists x) \psi_e(G,x)$.
        \item[(2)] If $c \Vdash (\forall x) \neg \psi_e(G,x)$ then $d \Vdash (\forall x) \neg \psi_e(G,x)$. 
    \end{itemize}
\end{lemma}
\begin{proof}
    Let $c := (\vec{R},\sigma, X)$ and $d := (\vec{R}\vec{S}, \sigma \cup \tau, Y)$. 
    Suppose $n = 0$.
    \begin{itemize}
        \item If  $c \Vdash (\exists x) \psi_e(G,x)$, then $(\exists x) \psi_e(\sigma,x)$.
        Moreover, $\sigma \cup \tau \succeq \sigma$, so by monotonicity of $\psi_e$ (see \Cref{rem:monotonicity-sigma01}), $(\exists x) \psi_e(\sigma \cup \tau,x)$, hence $d \Vdash (\exists x) \psi_e(G,x)$.
        
        %hence, since $\sigma \prec G$, by \Cref{rem:monotonicity-sigma01}), $(\exists x)\psi_e(G, x)$ holds.

        \item  If $c \Vdash (\forall x) \neg \psi_e(G,x)$, then $(\forall \rho \subseteq  X)(\forall x)( \rho$ is $\vec{R}$-transitive $ \implies \neg \psi_e(\sigma \cup \rho,x))$. Then, pick $x \in \omega$ and $\rho \subseteq Y$. Then, $\tau \cup \rho \subseteq \tau \cup Y \subseteq X$. Moreover, since $d$ is a condition, if $\rho$ is $\vec{R}\vec{S}$-transitive, $\sigma \cup \tau \cup \rho$ is $\vec{R}\vec{S}$-transitive, and in particular, $\tau \cup \rho$ is $\vec{R}$-transitive, hence $\neg \psi_e(\sigma \cup \tau \cup \rho, x)$. This yields that $(\forall \rho \subseteq  Y)(\forall x)( \rho$ is $\vec{R}\vec{S}$-transitive$ \implies \neg \psi_e(\sigma \cup \tau \cup \rho,x))$, i.e.,  $d \Vdash (\forall x) \neg \psi_e(G,x)$. 
    \end{itemize}
    Inductively, suppose $n > 0$.
    \begin{itemize}
        \item If  $c \Vdash (\exists x) \psi_e(G,x)$, then $(\exists x) c\Vdash \psi_e(\sigma,x)$, hence, by induction hypothesis, $(\exists x) d  \Vdash \psi_e(\sigma,x)$, i.e., $d \Vdash (\exists x) \psi_e(G,x)$.

        \item  If $c \Vdash (\forall x) \neg \psi_e(G,x)$, then $(\forall \rho \subseteq X)(\forall x)(\forall \ell \geq |\vec{R}| )(\rho$ is $\vec{R}$-transitive$ \implies \sigma  \cup \rho \qvdash_\ell \neg \psi_e(G,x))$. Then, pick $x \in \omega$, $\rho \subseteq Y$ and $\ell \geq |\vec{R}\vec{S}|$. Then, $\tau \cup \rho \subseteq \tau \cup Y \subseteq X$. Moreover, since $d$ is a condition, if $\rho$ is $\vec{R}\vec{S}$-transitive, $\sigma \cup \tau \cup \rho$ is $\vec{R}\vec{S}$-transitive, and in particular, $\tau \cup \rho$ is $\vec{R}$-transitive, thus, since $\ell \geq |\vec{R}|$, $\sigma \cup \rho \qvdash_\ell \neg \psi_e(G, x)$ holds. This yields that $(\forall \rho \subseteq  Y)(\forall x)(\forall \ell \geq \vec{R}\vec{S})( \rho$ is $\vec{R}\vec{S}$-transitive $ \implies \sigma  \cup \tau \cup \rho \qvdash_\ell \neg \psi_e(G,x))$, i.e.,  $d \Vdash (\forall x) \neg \psi_e(G,x)$. 
    \end{itemize}
    
\end{proof}

We now prove the core lemma for this notion of forcing: the forcing question meets its specifications. It implies in particular the density of the set of conditions forcing a property or its complement. Until now, the only hypothesis on the class~$\P$ was its partition regularity. Here, since we over-approximate the reservoir~$X$ by a class~$\L$ such that $\L \cap \UCNMN$ is large, one needs to assert some compatibility between~$\P$ and~$\L$ to deduce that~$X \in \L$. Since~$\L$ will be an intersection of~$\Sigma^0_1(\M_n)$ classes, assuming $\M_n$-minimality of~$\P$, that is, $\P \subseteq \langle \UCNMN \rangle$, we will have $X \in \P \subseteq \L$.

\begin{lemma}\label[lemma]{lem:question-below-validity}
     Let $n \in \omega$ and $c := (\vec{R},\sigma,X) \in \PP_\P$ such that $\P \subseteq \langle \UCNMN \rangle$. Consider $(\exists x) \psi_e(G,x)$ a $\Sigma_{n+1}^0$ formula. Let $m \geq |\vec{R}|$.
     \begin{itemize}
         \item If $\sigma \qvdash_m (\exists x) \psi_e(G,x)$ then $\exists d \leq c$ such that $d \Vdash (\exists x) \psi_e(G,x)$.
         \item If $\sigma \nqvdash_m (\exists x) \psi_e(G,x)$ then $\exists d \leq c$ such that $d \Vdash (\forall x) \neg \psi_e(G,x)$.
     \end{itemize}
\end{lemma}

\begin{proof}
    First, suppose $n =0$.
    \begin{itemize}
        \item Suppose $\sigma \qvdash_m (\exists x) \psi_e(G,x)$. Let $\L$ denote the class of all $Y$ such that
    for every~$m$-tuple of 2-colorings $\vec{h} = h_0, \dots, h_{m-1} \in 2^Y$ and every $\vec{S} \in \C_{m}(\sigma,Y)$, there is a finite $\tau \subseteq Y \setminus \{0, \dots, |\sigma|\}$ which is $\vec{S}$-transitive and $\vec{h}$-homogeneous, and some~$x \in \omega$ such that $ \psi_e(\sigma \cup \tau,x)$. Then, $\L \cap \U_{C_0}^{\M_0}$ is large.
    Since $\U_{C_0}^{\M_0}$ is $\M_0$-cohesive, $\langle \U_{C_0}^{\M_0} \rangle$ is $\M_0$-minimal. This yields that $\langle \U_{C_0}^{\M_0} \rangle \subseteq \L$. Moreover, $X \in \P  \subseteq \langle \U_{C_0}^{\M_0} \rangle \subseteq \L$. By a compactness argument, this yields that there exists $t \in \omega$ such that for every 2-colorings $\vec{h} = h_0, \dots, h_{m-1} \in 2^t$, there is a finite $\tau \subseteq X \cap \{0, \dots, t \}$ which is $\vec{R}$-transitive and $\vec{h}$-homogeneous and some $x \in \omega$ such that  $ \psi_e(\sigma \cup \tau,x)$. Let us build a specific $\vec{h}$ such that the $\tau$ we get gives us the extension of $c$ we look for. For every $i < m$, $a \leq t$, and $y \in X$, $y > t$, let $g_{i,a}(y) := 1$ if $aR_iy$, and 0 otherwise (if $m > i \geq |\vec{R}|$, then let~$R_i$ be a fixed dummy tournament). Since~$\P$ is partition regular, then there is some $\vec{g}$-homogeneous set~$H \subseteq X$ in~$\P$.  
    For every~$i < m$ and $a \leq t$, let $h_i(a) = 1$ if $\{a\} \to_{R_i} H$, and $0$ otherwise. Since~$X \in \L$, there exists a finite $\tau \subseteq X \cap \{0,\dots,t\}$ which is $\vec{R}$-transitive and $\vec{h}$-homogeneous and some $x \in \omega$ such that $\psi_e(\sigma\cup \tau,x)$ holds. Moreover, by \Cref{lem:extension}, $\vec{R} \in \C_{|\vec{R}|}(\sigma \cup \tau, H )$. This makes $d := (\vec{R},\sigma \cup \tau, H)$ a valid condition such that $d \Vdash (\exists x) \psi_e(G,x)$. 

    \item Suppose $\sigma \nqvdash_m (\exists x) \psi_e(G,x)$. Then, there exists $s \in \omega$ and $Y_0, \dots Y_s$ a partition of $\omega$ such that for all $i \leq s$, $(\dagger)$ either $Y_i \not \in \U_{C_0}^{\M_0}$ or there exists an $m$-tuple of 2-colorings $\vec{h} = h_0, \dots, h_{m-1} \in 2^{Y_i}$ and an $\vec{S} \in \C_{m}(\sigma,Y_i)$, such that for all $\tau \subseteq Y_i \setminus \{0, \dots, |\sigma|\}$ which is $\vec{S}$-transitive and $\vec{h}$-homogeneous, and for all~$x \in \omega$,  $ \neg \psi_e(\sigma \cup \tau,x)$. 

    Since~$\P$ is partition regular, then there is some~$i \leq s$ such that $X \cap Y_i \in \P$. In particular, $X \cap Y_i \in \U_{C_0}^{\M_0}$, so by upward-closure of partition regularity, $Y_i \in \U_{C_0}^{\M_0}$. Let~$\vec{h}$ and $\vec{S} \in \C_m(\sigma, Y)$ be witnesses of $(\dagger)$. By partition regularity of~$\P$, there is a $\vec{h}$-homogeneous subset~$H \subseteq X \cap Y_i$ in~$\P$. The condition $d := (\vec{R}\vec{S}, \sigma, H)$ is a valid extension of~$c$ such that $d \Vdash (\forall x) \neg \psi_e(G,x)$.
    \end{itemize}

    Now, inductively, suppose $n > 0$.
    \begin{itemize}
    
    \item Suppose $\sigma \qvdash_m (\exists x) \psi_e(G,x)$. Let $\L$ denote the class of all $Y$ such that
    for every~$m$-tuple of 2-colorings $\vec{h} = h_0, \dots, h_{m-1} \in 2^Y$ and every $\vec{S} \in \C_{m}(\sigma,Y)$, there is a finite $\tau \subseteq Y \setminus \{0, \dots, |\sigma|\}$ which is $\vec{S}$-transitive and $\vec{h}$-homogeneous, and some~$x, \ell \in \omega$ such that $\sigma \cup \tau \nqvdash_\ell\neg \psi_e(G, x)$. Then, $\L \cap \U_{C_n}^{\M_n}$ is large.
    Since $\U_{C_n}^{\M_n}$ is $\M_n$-cohesive, $\langle \U_{C_n}^{\M_n} \rangle$ is $\M_n$-minimal. This yields that $\langle \U_{C_n}^{\M_n} \rangle \subseteq \L$. Moreover, $X \in \P \subseteq \langle \UCNMN \rangle \subseteq \L$. By a compactness argument, this yields that there exists $t \in \omega$ such that for every $m$-tuple of 2-colorings $\vec{h} = h_0, \dots, h_{m-1} \in 2^t$,  there is a finite $\tau \subseteq X \cap \{0, \dots, t \}$ which is $\vec{R}$-transitive and $\vec{h}$-homogeneous and some $x, \ell \in \omega$ such that $\sigma \cup \tau \nqvdash_\ell \neg \psi_e(G, x)$. Let us build a specific $\vec{h}$ such that the $\tau$ we get gives us the extension of $c$ we look for. 
    
    For every $i < m$, $a \leq t$, and $y \in X$, $y > t$, let $g_{i,a}(y) := 1$ if $aR_iy$, and 0 otherwise. 
    
    By partition regularity of~$\P$, there is some $\vec{g}$-homogeneous set~$H \subseteq X$ in~$\P$.  
    
    For every~$i < m$ and $a \leq t$, let $h_i(a) = 1$ if $\{a\} \to_{R_i} H$, and $0$ otherwise. Now, there exists a finite $\tau \subseteq X \cap \{0,\dots,t\}$ which is $\vec{R}$-transitive and $\vec{h}$-homogeneous and some $x, \ell \in \omega$ such that $\sigma \cup \tau \nqvdash_\ell\neg \psi_e(G, x)$. Moreover, by \Cref{lem:extension}, $\vec{R} \in \C_m(\sigma \cup \tau, H) $. This makes $d := (\vec{R},\sigma \cup \tau, H)$ a valid condition such that for some~$\ell \geq m$, $\sigma \cup \tau \nqvdash_\ell\neg \psi_e(G, x)$, hence, by induction hypothesis, there exists $p \leq d \leq c$ such that $p \Vdash \psi_e(G, x)$, hence, $p \Vdash (\exists x) \psi_e(G,x)$.

    \item Suppose $\sigma \nqvdash_m (\exists x) \psi_e(G,x)$. Then, there exists $s \in \omega$ and $Y_0, \dots Y_s$ a partition of $\omega$ such that for all $i \leq s$, $(\dagger)$ either $Y_i \not \in \U_{C_n}^{\M_n}$ or there exists an $m$-tuple of 2-colorings $\vec{h} = h_0, \dots, h_{m-1} \in 2^{Y_i}$ and an $\vec{S} \in \C_{m}(\sigma,Y_i)$, such that for all $\tau \subseteq Y_i \setminus \{0, \dots, |\sigma|\}$ which is $\vec{S}$-transitive and $\vec{h}$-homogeneous, and for all~$x \in \omega$ and $\ell \geq m$, $\sigma \cup \tau \qvdash_\ell \neg \psi_e(G, x)$. 

    By partition regularity of~$\P$, there is some~$i \leq s$ such that $X \cap Y_i \in \P$. In particular, $X \cap Y_i \in \UCNMN$, so by upward-closure of partition regularity, $Y_i \in \UCNMN$. Let~$\vec{h}$ and $\vec{S} \in \C_m(\sigma, Y)$ be witnesses of $(\dagger)$. By partition regularity of~$\P$, there is a $\vec{h}$-homogeneous subset~$H \subseteq X \cap Y_i$ in~$\P$. 
    The condition $d := (\vec{RS}, \sigma , H)$ is a valid extension of~$c$ such that $d \Vdash (\forall x) \neg \psi_e(G,x)$.
    \end{itemize}
\end{proof}

\begin{definition}
    Let $n \in \omega$, and $\F \subseteq \PP_\P$ be a filter. The set $\F$ is said \emph{$n$-generic} if for all $k < n$, and every $\Sigma_{k+1}^0$ formula $(\exists x)\psi_e(G,x)$, there exists a condition $c \in \F$ such that $c \Vdash (\exists x)\psi_e(G,x)$ or $c \Vdash (\forall x)\neg \psi_e(G,x)$
\end{definition}

\begin{lemma}\label[lemma]{lem:n-genericity}
    Fix $n \in \omega$, and let $\F$ be a sufficiently generic $\PP_\P$-filter, where $\P \subseteq \langle \U^{\M_n}_{\C_n} \rangle$. Then $\F$ is $n$-generic.
\end{lemma}

\begin{proof}
    Let $k<n$, and let $(\exists x)\psi_e(G,x)$ be a $\Sigma_{k+1}$ formula. Let~$\D$ be the collection of all conditions deciding $(\exists x)\psi_e(G,x)$. We claim that~$\D$ is dense. Let $c := ( \vec{R},\sigma,X) \in \F$, and $m := |\vec{R}|$. Suppose $\sigma \qvdash_m (\exists x)(\psi_e(G,x)$. Then, by \Cref{lem:question-below-validity} there exists $d \leq c$ such that $d \Vdash (\exists x)(\psi_e(G,x)$. Otherwise, $\sigma \nqvdash_m (\exists x)(\psi_e(G,x)$. Then, by \Cref{lem:question-below-validity} there exists $d \leq c$ such that $d \Vdash (\forall x) \neg (\psi_e(G,x)$.
    Either way, there is some~$d \leq c$ in~$\D$, so $\D$ is dense. Since $\F$ is sufficiently generic, $\F \cap \D \neq \emptyset$.
\end{proof}

\begin{lemma}\label[lemma]{lem:non-contradiction}
    Let $n \in \omega$, and let $c := (\vec{R},\sigma,X) \in \PP_\P$ such that if $n > 0$ then $\P \subseteq \langle \U_{C_{n-1}}^{\M_{n-1}} \rangle$. Consider $(\exists x) \psi_e(G,x)$ a $\Sigma_{n+1}^0$ formula. Then, $(c \Vdash (\exists x) \psi_e(G,x)) \land (c \Vdash (\forall x) \neg \psi_e(G,x))$ never holds.
\end{lemma}

\begin{proof}
    We prove this inductively. Suppose otherwise.
    \begin{itemize}
        \item First, suppose $n=0$. Then, $(\exists x)\psi_e(\sigma,x)$ holds, and $(\forall \tau \subseteq X)(\forall y)( \tau$ is $\vec{R}$-transitive $ \implies \neg \psi_e(\sigma \cup \tau,y))$. In particular, for $\tau = \emptyset$, $\tau$ is $\vec{R}$-transitive, hence, $\neg \psi_e(\sigma,x)$ holds, yielding a contradiction.
        \item Now, suppose $n>0$. Then, $(\exists x)(c \Vdash \psi_e(G,x))$, and $(\forall \tau \subseteq X)(\forall y)(\forall \ell \geq |\vec{R}| )(\tau$ is $\vec{R}$-transitive$ \implies \sigma \cup \tau \qvdash_\ell \neg \psi_e(G,y))$. In particular, for $\tau = \emptyset$, and $\ell = \vec{R}$, $\tau$ is $\vec{R}$-transitive, hence, $\sigma \qvdash_\ell \neg \psi_e(G,x))$. Since $\P \subseteq \langle \U_{C_{n-1}}^{\M_{n-1}} \rangle$, this yields by \Cref{lem:question-below-validity} that there exists $d \leq c$ such that $d \Vdash \neg \psi_e(G,x)$. However, by \Cref{lem:forcing-relation-closure}, $d \Vdash \psi_e(G,x)$. This contradicts induction hypothesis.
    \end{itemize}
\end{proof}

The following lemma is known as the \qt{forcing implies truth} lemma: if a condition forces a formula, then for every sufficiently generic filter containing this condition, the formula will hold.

\begin{lemma}\label[lemma]{lem:forcing-implies-truth}
    Let $n \in \omega$, and $\F \subseteq \PP_\P$ be a filter such that if $n > 0$ then $\F$ is $(n-1)$-generic and $\P \subseteq \langle \U_{C_{n-1}}^{\M_{n-1}} \rangle$. Let $(\exists x)\psi_e(G,x)$ be a $\Sigma_{n+1}^0$ formula. Let $c \in \F$. 
    \begin{itemize}
        \item If $c \Vdash (\exists x)\psi_e(G,x)$, then $(\exists x)\psi_e(G_\F,x)$ holds.
        \item If $c \Vdash (\forall x) \neg \psi_e(G,x)$, then $ (\forall x) \neg \psi_e(G_\F,x)$ holds.
    \end{itemize}
\end{lemma}

\begin{proof}
    Suppose $n=0$.
    \begin{itemize}
        \item If $c \Vdash (\exists x)\psi_e(G,x)$, then $(\exists x)\psi_e(\sigma,X)$, and since $G_\F \succeq \sigma$, $(\exists x)\psi_e(G_\F,X)$.
        \item If $c \Vdash (\forall x) \neg \psi_e(G,x)$, then $(\forall \tau \subseteq X)(\forall x)( \tau$ is $\vec{R}$-transitive $ \implies \neg \psi_e(\sigma \cup \tau,x))$. Suppose for the contradiction that $(\exists x) \psi_e(G_\F,x)$. Then, by  \Cref{rem:monotonicity-sigma01}, there exists $\tau \subseteq G_\F \subseteq X$ such that $\psi_e(\sigma \cup \tau,x)$. However, by \Cref{lem:gf-transitive}, $G_\F$ is $\vec{R}$-transitive, hence $\tau$ is $\vec{R}$-transitive, contradicting hypothesis. 
    \end{itemize}
    Inductively, suppose $n > 0$.
      \begin{itemize}
        \item If $c \Vdash (\exists x)\psi_e(G,x)$, then $c \Vdash \psi_e(\sigma,x)$ for some~$x \in \omega$. By induction hypothesis, $(\exists x) \psi_e(G_\F,x)$.
        
        \item If $c \Vdash (\forall x) \neg \psi_e(G,x)$, then $(\forall \tau \subseteq X)(\forall x)(\forall \ell \geq |\vec{R}| )(\tau$ is $\vec{R}$-transitive$ \implies \sigma \cup \tau \qvdash_\ell \neg \psi_e(G,x))$. Fix some~$x \in \omega$.
        By $(n-1)$-genericity of~$\F$, there exists some~$d \in \F$ such that $d \Vdash \psi_e(G,x)$ or $d \Vdash \neg \psi_e(G,x)$. 
        By compatibility of the conditions in a filter, and by \Cref{lem:forcing-relation-closure}, we can suppose without loss of generality that $d \leq c$. In particular, $d := (\vec{R}\vec{S}, \sigma \cup \tau, Y)$. Since $\tau \subseteq X$ and is $\vec{R}$-transitive and since $|\vec{R}\vec{S}| \geq |\vec{R}|$, $\sigma \cup \tau \qvdash_{|\vec{R}\vec{S}|} \neg \psi_e(G,x))$.

        By \Cref{lem:question-below-validity}, there is some~$p \leq d$ such that $p \Vdash \neg \psi_e(G,x))$. Since~$p \leq d$ and $\P \subseteq \langle \U_{C_{n-1}}^{\M_{n-1}} \rangle$, then by \Cref{lem:non-contradiction}, $d \Vdash \neg \psi_e(G,x)$. By induction hypothesis, $\psi_e(G_\F, x)$ holds, and this, for every~$x$,
        so $(\forall x) \psi_e(G_\F, x)$.
    \end{itemize}
\end{proof}

\begin{lemma}\label[lemma]{lem:1-generic-infinite}
    Let $\F \subseteq \PP_\P$ be a 1-generic filter. Then $G_\F$ is infinite. 
\end{lemma}

\begin{proof}
    By 1-genericity of $\F$, for all $x$, there exists $c := (\vec{R}, \sigma, X) \in \F$ such that $c \Vdash (\exists y > x)( y \in G)$ or $c \Vdash (\forall y > x) (y \not \in G)$. Suppose the latter holds for some~$x$.

    Unfolding the definition of the forcing relation for $\Pi^0_1$ formulas, $(\forall \tau \subseteq X)(\forall y)(\tau \mbox{ is } \vec{R}\mbox{-transitive} \implies y \leq x \vee y \not \in \sigma \cup \tau)$.
    Since~$X$ is infinite, there is some~$y \in X$ such that $y > x$. Letting~$\tau = \{y\}$, $\tau$ is $\vec{R}$-transitive, $y > x$ and $y \in \sigma \cup \{y\}$. Contradiction.
\end{proof}

\section{Strong cone avoidance for arithmetic reductions}\label[section]{sect:strong-avoidance-arith}

We now use the framework developed in \Cref{sect:forcing-framework} to prove that $\EM$ admits strong cone avoidance for arithmetic reductions. In other words, the goal of this section is to prove our first main theorem:

\begin{repmaintheorem}{thm:arithmetic-main}
If $B$ is not arithmetic, then for every tournament~$T$, there is an infinite transitive subtournament~$H$ such that $B$ is not $H$-arithmetic.
\end{repmaintheorem}

Recall that in \Cref{sect:uc-sequence}, we stated the existence of an infinite sequence of Scott sets $\M_0 \subseteq \M_1 \subseteq \dots$ respectively coded by some sets~$M_0, M_1, \dots$, together with a sequence of sets $C_0, C_1, \dots$ such that $\emptyset^{(n+1)} \oplus C_n \in \M_{n+1}$, $M'_n \leq_T \emptyset^{(n+1)}$
and 
$$\langle \U_{C_0}^{\M_0} \rangle \supseteq \langle \U_{C_1}^{\M_1} \rangle \supseteq \dots$$
are partition regular classes. By \Cref{lem:intersection-still-large}, the class $\P_\omega = \bigcap_n \langle \U_{C_n}^{\M_n} \rangle$ is again partition regular.

\begin{definition}\label[definition]{def:omega-condition}
Let $\PP_\omega = \PP_{\P_\omega}$.
\end{definition}

Since $\P_\omega = \bigcap_n \langle \U_{C_n}^{\M_n} \rangle$, then all the hypothesis of the form $\P \subseteq \langle \U_{C_n}^{\M_n} \rangle$ hold in \Cref{sect:forcing-framework}. By \Cref{lem:question-below-complexity}, the forcing question to decide $\Sigma^0_n(G)$ formulas is $\Pi^0_1(\M_n)$, hence the definitional complexity of the forcing question is not at the same level in the hierarchy as the formula we force. Thankfully, in the case of arithmetic reductions, this difference is not relevant. Indeed, all the sets in~$\M_n$ are arithmetic, so if a set $B$ is not arithmetic, it is in particular not $\Pi^0_2(\M_n)$ for any~$n$.

\begin{lemma}\label[lemma]{lem:arith-diag}
Suppose $B$ is not arithmetic. 
Let $\F$ be a sufficiently generic $\PP_\omega$-filter. Then for every $n \in \omega$ and every $\Sigma^0_n$ formula $\varphi(G, x)$, there exists $d \in \F$ such that

$$ (\exists x \notin B)(d \Vdash \varphi(G, x)) \ \lor \   (\exists x \in B) (d \Vdash \neg \varphi(G, x)). $$
\end{lemma}

\begin{proof}
Fix some $c = (\vec{R}, \sigma, X) \in \F$, and let~$\varphi(G, x)$ be a $\Sigma^0_n$ formula for some~$n > 0$. Say~$m = |\vec{R}|$.
Let $W = \{ x :   \sigma \qvdash_m \varphi(G, x) \}$. By \cref{lem:question-below-complexity}, the set $W$ is $\Pi^0_n(\M_n)$. Since $B$ is not arithmetic, $W \neq B$. Let $x \in W \Delta B = (W \setminus B) \cup ( B \setminus W)$. One of the two cases holds:
\begin{itemize}
    \item $x \in W \setminus B$, then, by \cref{lem:question-below-validity}, there exists a condition $d \leq c$ such that $d \Vdash \varphi(G, x)$.
    \item $x \in B \setminus W$, then, by \cref{lem:question-below-validity},   there exists a condition $d \leq c$ such that $d \Vdash \neg \varphi(G, x)$.
\end{itemize}
In both cases, by genericity of~$\F$, there is such a~$d$ in~$\F$.
\end{proof}

We are now ready to prove our first main theorem.

\begin{proof}[Proof of \Cref{thm:arithmetic-main}]
Fix a non-arithmetic set~$B$, a tournament~$T$, and let~$\F$ be a sufficiently generic $\PP_\omega$-filter containing the condition $(T, \emptyset, \omega)$. By \Cref{lem:n-genericity}, $\F$ is $n$-generic for every~$n \in \omega$. By \Cref{lem:gf-transitive}, $G_\F$ is $T$-transitive, and by \Cref{lem:1-generic-infinite}, $G_\F$ is infinite.

We claim that $B$ is not $G_\F$-arithmetic:
Fix a $\Sigma^0_n$ formula $\varphi(G, x)$. By \Cref{lem:arith-diag},
there exists some~$c \in \F$ such that 
$$ (\exists x \notin B)(c \Vdash \varphi(G, x)) \ \lor \   (\exists x \in B) (c \Vdash \neg \varphi(G, x)). $$
By \Cref{lem:forcing-implies-truth}, $(\exists x \notin B)\varphi(G_\F, x) \lor  (\exists x \in B) \neg \varphi(G_\F, x)$, hence $B$ is not $G_\F$-arithmetic.
\end{proof}

\section{Layerwise strong cone avoidance}\label[section]{sect:layerwise-avoidance}

In this section, we are going to twist the previous notion of forcing to obtain a layerwise version of \Cref{thm:arithmetic-main}. More precisely, the goal of this section is to prove the following theorem:

\begin{repmaintheorem}{thm:layerwise-main}
Fix $n \geq 1$. If $B$ is not $\Sigma^0_n$, then for every tournament~$T$, there is an infinite transitive subtournament~$H$ such that $B$ is not $\Sigma^0_n(H)$.
\end{repmaintheorem}

As explained in \Cref{sect:picture-forcing-question}, the proof of such theorems is closely related to the existence of a uniformly $\Sigma^0_n$-preserving forcing question, that is, a forcing question for~$\Sigma^0_n(G)$ formulas which is $\Sigma^0_n$ uniformly in its parameters.

Unfortunately, by \Cref{lem:question-below-complexity}, forcing a $\Sigma^0_n(G)$ formula is $\Pi^0_1(\M_n)$, which is not the desired definitional complexity. We are going to use the same trick as Monin and Patey~\cite{monin2021weakness} and define a twisted notion of forcing \qt{on the top}, that is, leaving the lower levels unchanged, we are going to replace the Scott set~$\M_n$ by another Scott set~$\Nc_n$ with more suited properties, and define a different forcing question on the top level.

Intuitively, the bad complexity of the forcing question comes from the fact that, since there is no effectiveness restriction on the reservoir~$X$, the only way to decide properties is to check whether the class of sets satisfying this property is large. Largeness of a $\Sigma^0_n$ property is $\Pi^0_{n+1}$. Therefore, at the top level, the forcing question will have to directly involve the reservoir. The counterpart is that the forcing conditions will need to impose effectiveness restrictions on the reservoirs, which will raise a few technical difficulties.

\subsection{Top Scott set}

Suppose $B$ is a non-$\Sigma^0_{n+1}$ set. The forcing question on the top for~$\Sigma^0_{n+1}$ formulas will be~$\Sigma^0_1(\M_n)$, so for \Cref{prop-uniform-preservation-sigman} about diagonalization to work, one needs~$B$ not to be $\Sigma^0_1(\M_n)$. We will therefore replace the Scott set~$\M_n$ with another Scott set~$\Nc_n$ coded by a set~$N_n$ with the following two properties:
\begin{itemize}
    \item $\emptyset^{(n)}$ is coded by an element of~$\Nc_n$ ;
    \item $B$ is not $\Sigma^0_1(\Nc_n)$ 
\end{itemize}
The second fact replaces the previous assumption that $N_n'$ is computable in~$\emptyset^{(n+1)}$.
The existence of such a Scott set follows from the following proposition by Wang~\cite[Theorem 3.6]{wang2016definability}:

\begin{proposition}[Wang~\cite{wang2016definability}]\label[proposition]{prop:wang}
Let~$Z, B$ such that $B$ is not $\Sigma^0_1(Z)$. For every~$Z$-computable tree~$T \subseteq 2^ {<\NN}$, there exists an infinite path~$P \in [T]$ such that $B$ is not $\Sigma^0_1(Z \oplus P)$.
\end{proposition}

However, the two properties above are not sufficient to define~$\Nc_n$. Indeed, there are still no effectiveness restrictions on the initial tournament~$T$, and in particular, $T$ cannot be assumed to be in~$\Nc_n$, but in the proof of \Cref{lem:question-below-validity}, one needs to split the reservoir~$X$ based on a $Z \oplus T$-computable finite partition, for some~$Z \in \Nc_n$. The resulting reservoir must still belong to~$\Nc_n$ and to the partition regular class~$\P$. The Scott set~$\Nc_n$ must therefore enjoy the following property:
\begin{itemize}
    \item For every $X \in \Nc_n \cap \P$, every~$Z \in \Nc_n$ and every~$T \oplus Z$-computable set~$A$,
there exists an infinite set $Y \subseteq X \cap A$ or $Y \subseteq X \cap \overline{A}$
such that $Y \in \Nc_n \cap \P$.
\end{itemize}

For this, we will prove the following proposition, which is an adaptation of an alternative proof by Hirschfeldt~\cite[Lemma~6.63]{hirschfeldt2017slicing} of a theorem by Dzhafarov and Jockusch~\cite{jockusch2009ramsey}. The statement of the proposition can be found in Monin and Patey~\cite[Theorem 5.1]{monin2022partition} and \cite[Theorem 4.11]{monin2021weakness}, but without the assumption that $H \in \U^Z_C$. We therefore give a direct proof of it for the sake of simplicity.

\begin{proposition}\label[proposition]{prop:dzhafarov-jockusch}
Let~$Z, B$ such that $B$ is not $\Sigma^0_1(Z)$. 
Let~$\U^Z_C$ be a partition regular class.
For every set~$A$, there exists an infinite subset~$H \subseteq A$ or $H \subseteq \overline{A}$ such that $B$ is not $\Sigma^0_1(Z \oplus H)$ and $H \in \U^Z_C$.
\end{proposition}

Fix~$Z, B, \U^Z_C$ and~$A$. Say~$A_0 = A$ and $A_1 = \overline{A}$. We are going to build two sets~$G_0, G_1$ by a variant of Mathias forcing whose conditions are tuples of the form $(\sigma_0, \sigma_1, X)$, where
\begin{itemize}
    \item $(\sigma_i, X)$ is a Mathias condition with~$\sigma_i \subseteq A_i$
    \item $X \in \U^Z_C$ and $B$ is not $\Sigma^0_1(X \oplus Z)$
\end{itemize}
Consider the following two kind of requirements, for every~$e \in \omega$ and~$k \in C$:
$$\R^G_e \mbox{: } W_e^{G \oplus Z} \neq B \hspace{20pt}
\S^G_k \mbox{: } G \in \U^Z_k$$
We are going to construct~$G_0, G_1$ such that they satisfy the following requirements for every~$e_0, e_1 \in \omega$ and~$k_0, k_1 \in C$: 
$$\R^{G_0}_{e_0} \vee \R^{G_1}_{e_1} \mbox{ ; } \S^{G_0}_{k_0} \vee \S^{G_1}_{k_1} \mbox{ ; } \R^{G_0}_{e_0} \vee \S^{G_1}_{k_1} \mbox{ ; } \S^{G_0}_{k_0} \vee \R^{G_1}_{e_1}$$ 
By a pairing argument, then there is some~$i < 2$ such that~$G_i \in \U^Z_C$ and $B$ is not~$\Sigma^0_1(G_i \oplus Z)$. We will need the following three lemmas. The fourth case follows by symmetry.

\begin{lemma}\label[lemma]{lem:prop-dj-rr}
Let~$c$ be a condition and $e_0, e_1 \in \omega$. There is an extension forcing~$\R^{G_0}_{e_0} \vee \R^{G_1}_{e_1}$.
\end{lemma}
\begin{proof}
Say~$c = (\sigma_0, \sigma_1, X)$. Let~$W$ be the $\Sigma^0_1(X \oplus Z)$ set of all~$x \in \omega$ such that for every 2-partition $R_0 \sqcup R_1 = X$, there is some~$i < 2$ and some $\rho \subseteq R_i$ such that $x \in W_{e_i}^{(\sigma_i \cup \rho) \oplus Z}$. Since~$W$ is $\Sigma^0_1(X \oplus Z)$ while~$B$ is not, there is some~$x \in W \Delta B = (W \setminus B) \cup (B \setminus W)$. 
\begin{itemize}
    \item If~$x \in W \setminus B$, then, letting~$R_i = X \cap A_i$, there is some~$i < 2$ and some~$\rho \subseteq X \cap A_i$ such that $x \in W_{e_i}^{(\sigma_i \cup \rho) \oplus Z}$. The condition $(\sigma_i \cup \rho, \sigma_{1-i}, X \setminus \{0, \dots, \max \rho \})$ is an extension of~$c$ forcing~$\R^{G_i}_{e_i}$.
    \item If~$x \in B \setminus W$, then, let~$\C$ be the $\Pi^0_1(X \oplus Z)$ class of all $R_0 \oplus R_1$ such that $R_0 \sqcup R_1 = X$, for every~$i < 2$ and every~$\rho \subseteq R_i$, $x \not\in W_{e_i}^{(\sigma_i \cup \rho) \oplus Z}$. The class~$\C$ is non-empty, so by \Cref{prop:wang}, there is some~$R_0 \oplus R_1 \in \C$ such that $B$ is not $\Sigma^0_1(R_0 \oplus R_1 \oplus X \oplus Z)$. By partition regularity of~$\U^Z_C$, there is some~$i < 2$ such that $R_i \in \U^Z_C$. The condition $(\sigma_0, \sigma_1, R_i)$ is an extension of~$c$ forcing~$\R^{G_i}_{e_i}$.
\end{itemize}
\end{proof}

\begin{lemma}\label[lemma]{lem:prop-dj-ss}
Let~$c$ be a condition and $k_0, k_1 \in C$. There is an extension forcing~$\S^{G_0}_{k_0} \vee \S^{G_1}_{k_1}$.
\end{lemma}
\begin{proof}
Say~$c = (\sigma_0, \sigma_1, X)$. Since~$\U^Z_C$ is partition regular and $X \in \U^Z_C$, there is some~$i < 2$ such that $X \cap A_i \in \U^Z_C$. In particular, $X \cap A_i \in \U^Z_{k_i}$. Let~$\rho \subseteq X \cap A_i$ be such that $\rho \in \U^Z_{k_i}$.
Then the condition $(\sigma_i \cup \rho, \sigma_{1-i}, X \setminus \{0, \dots, \max \rho \})$ is an extension of~$c$ forcing~$\S^{G_i}_{k_i}$.
\end{proof}

\begin{lemma}\label[lemma]{lem:prop-dj-rs}
Let~$c$ be a condition, $e_0 \in \omega$ and $k_1 \in C$. There is an extension forcing~$\R^{G_0}_{e_0} \vee \S^{G_1}_{k_1}$.
\end{lemma}
\begin{proof}
Say~$c = (\sigma_0, \sigma_1, X)$. Let~$W$ be the $\Sigma^0_1(X \oplus Z)$ set of all~$x \in \omega$ such that for every 2-partition $R_0 \sqcup R_1 = X$, either there is some $\rho \subseteq R_0$ such that $x \in W_{e_0}^{(\sigma_0 \cup \rho) \oplus Z}$, or there is some~$\rho \subseteq R_1$ such that $\rho \in \U^Z_{k_1}$. Since~$W$ is $\Sigma^0_1(X \oplus Z)$ while~$B$ is not, there is some~$x \in W \Delta B = (W \setminus B) \cup (B \setminus W)$. 
\begin{itemize}
    \item If~$x \in W \setminus B$, then, letting~$R_i = X \cap A_i$, either there is some~$\rho \subseteq X \cap A_0$ such that $x \in W_{e_0}^{(\sigma_0 \cup \rho) \oplus Z}$, or some~$\rho \subseteq X \cap A_1$ such that $\rho \in \U^Z_{k_1}$. The condition $(\sigma_i \cup \rho, \sigma_{1-i}, X \setminus \{0, \dots, \max \rho \})$ is an extension of~$c$ forcing~$\R^{G_0}_{e_0}$ in the first case, and $\S^{G_1}_{k_1}$ in the second case.
    \item If~$x \in B \setminus W$, then, let~$\C$ be the $\Pi^0_1(X \oplus Z)$ class of all $R_0 \oplus R_1$ such that for every~$\rho \subseteq R_0$, $x \not\in W_{e_0}^{(\sigma_i \cup \rho) \oplus Z}$, and $R_1 \not \in \U^Z_{k_1}$. The class~$\C$ is non-empty, so by \Cref{prop:wang}, there is some~$R_0 \oplus R_1 \in \C$ such that $B$ is not $\Sigma^0_1(R_0 \oplus R_1 \oplus X \oplus Z)$. By partition regularity of~$\U^Z_C$, there is some~$i < 2$ such that $R_i \in \U^Z_C$, and by choice of~$\C$, $i = 0$. The condition $(\sigma_0, \sigma_1, R_0)$ is an extension of~$c$ forcing~$\R^{G_0}_{e_0}$.
\end{itemize}
\end{proof}

We are now ready to prove \Cref{prop:dzhafarov-jockusch}.

\begin{proof}[Proof of \Cref{prop:dzhafarov-jockusch}]
Fix~$Z, B, \U^Z_C$ and~$A$. Say~$A_0 = A$ and $A_1 = \overline{A}$.
Let~$\F$ be a sufficiently generic filter for this notion of forcing.
For every~$i < 2$, let $G_i = \bigcup_{(\sigma_0, \sigma_1, X) \in \F} \sigma_i$.
By definition of a condition, $G_0 \subseteq A$ and $G_1 \subseteq \overline{A}$.
By \Cref{lem:prop-dj-rr}, \Cref{lem:prop-dj-ss}, \Cref{lem:prop-dj-rs}
and its symmetric version, there is some~$i < 2$ such that~$G_i \in \U^Z_C$ and $B$ is not~$\Sigma^0_1(G_i \oplus Z)$. In particular, assuming~$\U^Z_C$ contains only infinite sets, $G_i$ is infinite.
This completes the proof of \Cref{prop:dzhafarov-jockusch}.
\end{proof}

Thanks to \Cref{prop:wang} and \Cref{prop:dzhafarov-jockusch}, one can prove the existence of a Scott set~$\Nc_n$ satisfying the properties mentioned above:

\begin{proposition}\label[proposition]{prop:top-extension}
Fix~$n > 0$.
Let~$B$ be a non-$\Sigma^0_{n+1}$ set and $T$ be a tournament. 
There exists a Scott set~$\Nc_n$ such that 
\begin{itemize}
    \item $\emptyset^{(n)} \in \Nc_n$ ; $B$ is not $\Sigma^0_1(\Nc_n)$ ;
    \item for every $X \in \Nc_n \cap \langle \U_{C_{n-1}}^{\M_{n-1}}\rangle$, every~$Z \in \Nc_n$ and every~$T \oplus Z$-computable set~$A$, there exists an infinite set $Y \subseteq X \cap A$ or $Y \subseteq X \cap \overline{A}$ such that $Y \in \Nc_n \cap \langle \U_{C_{n-1}}^{\M_{n-1}}\rangle$.
\end{itemize}
\end{proposition}
\begin{proof}
By \Cref{prop:wang} and \Cref{prop:dzhafarov-jockusch}, there exists an infinite sequence~$Z_0, Z_1, \dots$ such that $Z_0 = \emptyset^{(n)}$, and for every~$s \in \omega$:
\begin{itemize}
    \item[(1)] For every~$Z_0 \oplus \dots \oplus Z_s$-computable infinite binary tree~$T$, there is some~$t$ such that $Z_t \in [T]$ ;
    \item[(2)] For every~$Z_0 \oplus \dots \oplus Z_s$-computable infinite set~$X \in \langle \U_{C_{n-1}}^{\M_{n-1}}\rangle$ and every~$T \oplus Z_0 \oplus \dots \oplus Z_s$-computable set~$A$, there exists some~$t$ such that $Z_t \subseteq X \cap A$ or $Z_t \subseteq X \cap \overline{A}$ and $Z_t \in \Nc_n \cap \langle \U_{C_{n-1}}^{\M_{n-1}}\rangle$.
    \item[(3)] $B$ is not $\Sigma^0_1(Z_0 \oplus \dots \oplus Z_s)$
\end{itemize}
Let~$\Nc_n = \{ X : (\exists s) X \leq_T Z_0 \oplus \dots \oplus Z_s \}$. By construction, $\Nc_n$ is a Turing ideal containing~$\emptyset^{(n)}$. Moreover, by (1), $\Nc_n \models \WKL$, by (2), $\Nc_n$ satisfies the second item of the lemma, and by (3), $B$ is not $\Sigma^0_1(\Nc_n)$.
\end{proof}

As explained in \Cref{sect:minimal-cohesive-classes}, given a Scott set~$\M$ coded by a set~$M$,
one can compute the index set~$C$ of an $\M$-cohesive class~$\U^\M_C$ in any PA degree over~$M'$ (see Monin and Patey~\cite[Lemma 2.17]{monin2021weakness} for a full proof). 
Since~$M_{n-1}' \leq_T \emptyset^{(n)}$ and $\emptyset^{(n)} \in \Nc_n$ which is a Scott set, then one can find the index set~$C_{n-1}$ of an $\M_{n-1}$-cohesive large class $\U^{\M_{n-1}}_{C_{n-1}}$ in~$\Nc_n$.

\subsection{Top forcing conditions}

The notion of forcing for layerwise cone avoidance resembles the previous notion of forcing, with a few distinctive features. 
\begin{itemize}
    \item First, since one needs to control only~$\Sigma^0_{n+1}(G)$ properties, the partition generic class~$\P$ will only need to be included in a minimal class of a finite level of the hierarchy of Scott sets. We will actually choose $\P = \langle \U_{C_{n-1}}^{\M_{n-1}}\rangle$.
    \item  Second, since the forcing question on the top will depend on the reservoir~$X$, one must require that $X \in \Nc_n$, in order to obtain a $\Sigma^0_1(\Nc_n)$ forcing question. Since ~$B$ is not $\Sigma^0_1(\Nc_n)$, the diagonalization lemma will hold.
    \item Last, as explained above, since the proof of \Cref{lem:question-below-validity} splits the reservoir based on 2-partitions computable in the tournaments, one must require that $\Nc_n$ is closed under this operation. By construction of~$\Nc_n$, the closure is ensured for the original tournament~$T$. However, in the proof of \Cref{lem:question-below-validity}, new tournaments will be added to the condition. Thankfully, all the new tournaments can be chosen as members of $\Pi^0_1(\Nc_n)$ classes, and since~$\Nc_n$ is a Scott set, one can require that~$\vec{R} \in \Nc_n$. The new notion of condition will therefore distinguish between the original tournament~$T$ which can be of arbitrary strength, and the new tournaments~$\vec{R}$ added along the construction of a generic filter, and which will belong to~$\Nc_n$.
\end{itemize} 

\begin{definition}\label[definition]{def:top-condition}
    Fix a tournament~$T$.
    Let $n > 0$. Let $\PP_n$ denote the set of all 3-tuples $(\vec{R},\sigma,X)$ such that
    \begin{enumerate}
        \item \vr~is a finite sequence of tournaments,
        \item $X \cap \{0, \dots, |\sigma| \} = \emptyset$, 
        \item $X \in \langle \U_{C_{n-1}}^{\M_{n-1}}\rangle$,
        \item $\vec{R}, X \in \Nc_n$,
        \item for all $y \in X$, $\sigma \cup \{y\}$ is \vr-transitive and $T$-transitive,
        \item $X$ is included in a minimal $\vec{R}$-interval and $T$-interval of~$\sigma$.
        
    \end{enumerate}
    In other words, $\PP_n$ is the set of all $\PP_\P$-conditions $((T,\vec{R}), \sigma, X)$ for~$\P = \langle \U_{C_{n-1}}^{\M_{n-1}}\rangle$ such that $\vec{R}, X \in \Nc_n$.
\end{definition}

\begin{remark}
    As an element of $\PP_\P$, a $\PP_n$ condition inherits the definitions of the forcing relation and the forcing question. The Scott set~$\Nc_n$ has been designed so that the proof of \Cref{lem:question-below-validity} still holds while ensuring that $\vec{R}$ and $X$ belong to~$\Nc_n$. %\ludovic{TODO, étoffer la remarque, pour dire que la forcing question fonctionne aussi avec $\PP_n$, dans la mesure où dans le cas oui, la couverture qui witness la failure de largeness appartient à la classe $\M_{n-1} \subseteq \Nc_n$, et par la propriété de $\Nc_n$, les coloriages $\vec{g}$ étant $T \oplus Z$-calculables pour un $Z \in \Nc_n$, on trouve un sous-ensemble $\vec{g}$-homogène $Y$ de $X$ dans $\Nc_n$. Dans le cas non, la classe des  tournois et coloriages qui witnessent la propriété est $\Pi^0_1(\Nc_n)$, et donc comme $\Nc_n \models \WKL$, on en trouve des témoins dans $\Nc_n$.}
\end{remark}

\subsection{Top forcing question}

We now define a forcing question which is very similar to \Cref{def:forcing-question-em-level0}.

\begin{definition}
Fix $n > 0$. Let $c = (\vec{R}, \sigma, X) \in \PP_n$ and let $(\exists x)\psi_e(G, x)$ be a $\Sigma^0_{n+1}$ formula. Say $m := |\vec{R}|+1$. Define the relation $c \qvdash (\exists x)\psi_e(G, x)$ to hold if for every~$m$-tuple of 2-colorings $\vec{h} = h_0, \dots, h_{m-1} \in 2^X$ and every $\vec{S} \in \C_{m}(\sigma,X)$, there is an $\vec{h}$-homogeneous and $\vec{S}$-transitive $\tau \subseteq X$ and some $x \in \omega$ such that $\sigma \cup \tau \nqvdash_m \neg \psi_e(G,x)$.   

\end{definition}

\begin{lemma}\label[lemma]{lem:question-complexity}

Fix $n > 0$. Let $c = (\vec{R}, \sigma, X) \in \PP_n$ and let $(\exists x)\psi_e(G, x)$ be a $\Sigma^0_{n+1}$ formula. The sentence $(c \qvdash (\exists x)\psi_e(G, x))$ is $\Sigma^0_1(\Nc_n)$.

\end{lemma}

\begin{proof}

Let $m := |\vec{R}|+1$. By a compactness argument, $(c \qvdash (\exists x)\psi_e(G, x))$ holds if there exists $t \in \omega$ such that for every~$m$-tuple of 2-colorings $\vec{h} = h_0, \dots, h_{m-1} \in 2^t$ and every every $m$-tuple of tournaments $\vec{S}$ over $\{0, \dots, t\}$ such that $\vec{S} \in \C_{m}(\sigma,X \cap \{0, \dots, t\})$, there is an $\vec{h}$-homogeneous and  $\vec{S}$-transitive $\tau \subseteq X \cap \{0, \dots ,t \}$ and some $x \in \omega$ such that $\sigma \cup \tau \nqvdash_m \neg \psi_e(G,x)$. The formula $\neg \psi_e(G,x)$ is $\Sigma^0_n$, so by \cref{lem:question-below-complexity}, the formula 
$(\sigma \cup \tau \qvdash_m \neg \psi_e(G,x))$ is $\Pi^0_1(C_{n-1} \oplus \emptyset^{(n)})$ uniformly in its parameters, hence $(\sigma \cup \tau \nqvdash_m \neg \psi_e(G,x))$ is $\Sigma^0_1(\M_n)$. This yields the expected result since~$\M_n \subseteq \Nc_n$.

\end{proof}

The following lemma states that the forcing question on the top meets its specifications, that is, based on its answer, there is an extension forcing the property or its complement.

\begin{lemma}\label[lemma]{lem:question-validity}
     Let $ n \in \omega$ and $c := (\vec{R},\sigma,X) \in \PP_n$. Consider $(\exists x) \psi_e(G,x)$ a $\Sigma_{n+1}^0$ formula.
     \begin{itemize}
         \item If $c \qvdash (\exists x) \psi_e(G,x)$ then $\exists d \leq c$ such that $d \Vdash (\exists x) \psi_e(G,x)$.
         \item If $c \nqvdash (\exists x) \psi_e(G,x)$ then $\exists d \leq c$ such that $d \Vdash (\forall x) \neg \psi_e(G,x)$.
     \end{itemize}
\end{lemma}

\begin{proof}
Let $m = |\vec{R}|+1$.
For simplicity of notation, let $R_{m-1} = T$.

First, suppose $c \qvdash (\exists x) \psi_e(G,x)$. Then, in particular, for every~$m$-tuple of 2-colorings $\vec{h} = h_0, \dots, h_{m-1} \in 2^X$, there exists an $\vec{h}$-homogeneous, $\vec{R}$-transitive and $T$-transitive $\tau \subseteq X$ and some $x \in \omega$ such that $\sigma \cup \tau \nqvdash \neg \psi_e(G,x)$. By a compactness argument, there exists $t \in \omega$ such that we can restrict the considered set of $m$-tuples of $2$-colorings of $X$ to 2-colorings of $\{0,\dots, t\}$. We again build the same $m$-tuple of 2-colorings as follows:
for every $i < m$, $a \leq t$, and $y \in X$, $y > t$, let $g_{i,a}(y) := 1$ if $R_i(a,y)$ holds, and 0 otherwise. Note that $\vec{g}$ is $T \oplus \vec{R} \oplus X$-computable, hence $T \oplus Z$-computable for some~$Z \in \Nc_n$ (since~$R_{m-1} = T$).
    
     By choice of~$\Nc_n$, with $\P = \langle \U_{C_{n-1}}^{\M_{n-1}} \rangle$, by \Cref{prop:top-extension}, there exists $H \subseteq X$ a $\vec{g}$-homogeneous set in $\Nc_n \cap \langle \U_{C_{n-1}}^{\M_{n-1}}\rangle$. 
    
    For every~$i < m$ and $a \leq t$, let $h_i(a) = 1$ if $\{a\} \to_{R_i} H$, and $0$ otherwise. Now, there exists a finite $\tau \subseteq X \cap \{0,\dots,t\}$ which is $\vec{R}$-transitive, $T$-transitive and $\vec{h}$-homogeneous and some $x \in \omega$ such that $\sigma \cup \tau \nqvdash_m \neg \psi_e(G, x)$. 
    Moreover, by \Cref{lem:extension}, $(\vec{R},T) \in \C_m(\sigma \cup \tau, H) $. This makes $d := (\vec{R},\sigma \cup \tau, H)$ a valid $\PP_n$ condition under $c$ such that $d \nqvdash_m \neg \psi_e(G, x)$, hence, by \Cref{lem:question-below-validity}, there exists $p \leq d \leq c$ such that $p \Vdash \psi_e(G,x)$.

Now, suppose $c \nqvdash (\exists x) \psi_e(G,x)$. Then, there exists an $m$-tuple of 2-colorings $\vec{h} = h_0, \dots, h_{m-1} \in 2^X$ and an $m$-tuple of tournaments $\vec{S} \in \C_m(\sigma,X)$ such that for every $\vec{h}$-homogeneous and $\vec{S}$-transitive finite chain $\sigma \subseteq X$ and every $x \in \omega$, $\sigma \cup \tau \qvdash_m \neg \psi_e(G,x)$. Let $\L$ be the class of such $m$-tuples of 2 colorings $\vec{h}$ and $m$-tuples of $\vec{S} \in \C_m(\sigma,X)$. By \Cref{lem:approx-pi01}, $\C_m(\sigma, X)$ is $\Pi^0_1(X)$. Since $X \in \Nc_n$, the class~$\L$ is $\Pi_1^0(\Nc_n)$, hence, since $\Nc_n \models \WKL$, there exists $(\vec{h}, \vec{S}) \in \Nc_n \cap \L$. By partition regularity of $\P = \langle \U_{C_{n-1}}^{\M_{n-1}} \rangle$, there is a $\vec{h} \oplus X$-computable $\vec{h}$-homogeneous set~$Y \subseteq X$ in~$\P$.
The 3-tuple $d := (\vec{R}\vec{S}, \sigma, Y)$ is a valid $\EM$-condition such that $d \Vdash (\forall x) \neg \psi_e(G,x)$.

\end{proof}

The following diagonalization lemma is a specialization of \Cref{prop-uniform-preservation-sigman} to this notion of forcing.

\begin{lemma}\label[lemma]{lem:layerwise-diag}
  Fix $n > 0$. Let $\F$ be a sufficiently generic $\PP_n$-filter. Then, for every $\Sigma_{n+1}^0$ formula $\varphi(G,x)$, there exists $d \in \F$ such that
    $$ (\exists x \notin B)(d \Vdash  \varphi(G, x)) \ \lor \   (\exists x \in B) (d \Vdash \neg \varphi(G, x)). $$
\end{lemma}

\begin{proof}
Fix some $c = (\vec{R}, \sigma, X) \in \F$, and let~$\varphi(G, x)$ be a $\Sigma^0_{n+1}$ formula.
Let $W = \{ x :   c \qvdash (\varphi (G, x) \}$. By \cref{lem:question-complexity}, the set $W$ is $\Sigma^0_1(\Nc_n)$. By construction of $\Nc_n$ in \Cref{prop:top-extension}, $B$ is not $\Sigma_1^0(\Nc_n)$, hence, $W \neq B$. Let $x \in W \Delta B = (W \setminus B) \cup ( B \setminus W)$. One of the two cases holds:
\begin{itemize}
    \item $x \in W \setminus B$, then, by \cref{lem:question-validity}, there exists a condition $d \leq c$ such that $d \Vdash \varphi(G, x)$.
    \item $x \in B \setminus W$, then, by \cref{lem:question-validity},   there exists a condition $d \leq c$ such that $d \Vdash \neg \varphi(G, x)$.
\end{itemize}
In both cases, by genericity of~$\F$, there is such a~$d$ in~$\F$.
\end{proof}

We are now ready to prove \Cref{thm:layerwise-main}.

\begin{proof}[Proof of \Cref{thm:layerwise-main}]
The case $n = 0$ is proven independently by the first author and Wang (unpublished) and is a consequence of \Cref{sect:combinatorics-em}.
Fix $n>0$, and fix a non-$\Sigma_{n+1}^0$ set~$B$, a tournament~$T$, and let~$\F$ be a sufficiently generic $\PP_n$-filter containing the condition $(\emptyset, \emptyset, \omega)$ (recall that $T$ is a parameter of the notion of forcing). By \Cref{lem:n-genericity}, $\F$ is $(n-1)$-generic. By \Cref{lem:gf-transitive}, $G_\F$ is $T$-transitive, and by \Cref{lem:1-generic-infinite}, $G_\F$ is infinite.

We claim that $B$ is not $\Sigma^0_{n+1}(G_\F)$:
fix a $\Sigma^0_{n+1}$ formula $\varphi(G, x)$. By \Cref{lem:layerwise-diag},
there exists some~$c \in \F$ such that 
$$ (\exists x \notin B)(c \Vdash \varphi(G, x)) \ \lor \   (\exists x \in B) (c \Vdash \neg \varphi(G, x)). $$
By \Cref{lem:forcing-implies-truth}, since $\F$ is $(n-1)$-generic, $(\exists x \notin B)\varphi(G_\F, x) \lor  (\exists x \in B) \neg \varphi(G_\F, x)$, hence $B$ is not  $\Sigma^0_n(G_\F)$.
\end{proof}

\section{Effective constructions and lowness}\label[section]{sect:effective-constructions}

This last section of our article is devoted to the proof of the third main theorem:

\begin{repmaintheorem}{thm:effective-main}
Fix $n \geq 1$. Every $\Delta^0_n$ tournament~$T$ has an infinite transitive subtournament of low${}_{n+1}$ degree.
\end{repmaintheorem}

First of all, notice that this bound is tight, in that there exists a computable tournament with no infinite $\Sigma^0_2$ transitive subtournament (see  Patey~\cite{patey2015somewhere}). By relativizing the argument, for every~$n \geq 1$, there is a $\Delta^0_n$ tournament with no infinite $\Sigma^0_{n+1}$ transitive subtournament, hence no infinite low${}_n$ transitive subtournament.
We will actually prove the following stronger theorem:

\begin{theorem}\label[theorem]{thm:effective-strong-main}
Fix $n \geq 1$. For every set~$P$ of PA degree over $\emptyset^{(n)}$, every $\Delta^0_n$ tournament~$T$ has an infinite transitive subtournament~$H$ such that $H^{(n)} \leq_T P$.
\end{theorem}

\Cref{thm:effective-main} follows from \Cref{thm:effective-strong-main} using the low basis theorem:

\begin{proof}[Proof of \Cref{thm:effective-main}]
Fix $n \geq 1$ and a $\Delta^0_n$ tournament~$T$. By the low basis theorem relative to~$\emptyset^{(n)}$ (see Jockusch and Soare~\cite{jockusch197classes}), there is a set~$P$ of PA degree over~$\emptyset^{(n)}$ such that $P' \leq_T \emptyset^{(n+1)}$. By \Cref{thm:effective-strong-main}, there is an infinite $T$-transitive subtournament~$H$ such that $H^{(n)} \leq_T P$. In particular, $H^{(n+1)} \leq_T P' \leq_T \emptyset^{(n+1)}$, hence $H$ is of low${}_{n+1}$ degree.
\end{proof}

The rest of this section is therefore devoted to the proof of \Cref{thm:effective-strong-main}.
The goal will be to construct, given a $\Delta^0_n$ tournament~$T$ and a set~$P$ of PA degree over $\emptyset^{(n)}$, an infinite decreasing sequence of conditions $c_0 = (T, \emptyset, \omega) \geq c_1 \geq \dots$ such that the induced filter~$\F = \{ d : (\exists n) c_n \leq d \}$ is $n$-generic. Then, the set~$G_\F$ be will be an infinite $T$-transitive subtournament such that $G_\F^{(n)} \leq_T P$.

We will work with a notion of forcing~$\QQ_n$ which is very similar to~$\PP_n$, with the same distinction between the forcing question of the top and the ones at the lower levels. The main difference between~$\QQ_n$ and~$\PP_n$ comes from two facts:
\begin{itemize}
    \item The tournament~$T$ is $\Delta^0_n$, hence belongs to~$\M_n$. There is therefore no need to distinguish the original tournament~$T$ from the sequence of tournaments~$\vec{R}$ obtained with the question of forcing.
    \item The resulting filter will be $n$-generic, but there will be no diagonalization against a fixed non-$\Sigma^0_{n+1}$ set~$B$. We therefore does not require that a set~$B$ is not $\Sigma^0_1(\M_n)$.
\end{itemize}
Because of these differences, there is no need to use a different Scott set at the level~$n$. We will therefore keep~$\M_n$ instead of replacing it with~$\Nc_n$. Note that any PA degree over~$\emptyset^{(n)}$ can compute a sequence of sets~$M_0, M_1, \dots, M_n$ satisfying the properties of \Cref{mn-seq}, except that the last set~$M_n$ is not required to satisfy the last item, but simply to be $P$-computable.

\begin{definition}\label[definition]{def:top-condition}
    Let $n > 0$. Let $\QQ_n$ denote the set of all $\PP_\P$-conditions $(\vec{R}, \sigma, X)$ for~$\P = \langle \U_{C_{n-1}}^{\M_{n-1}}\rangle$ such that $\vec{R}, X \in \M_n$.
\end{definition}

In order to analyze the computational power needed to construct an $n$-generic decreasing sequence of conditions, one must fix a coding of the conditions into finite objects. Since~$\M_n = \{ Z_e : e \in \NN \}$ is countably coded by the set~$M_n = \bigoplus_{e \in \omega} Z_e$, every element~$X \in \M_n$ can be represented by an integer~$e$ such that $X = Z_e$. We call such an~$e$ an \emph{$M_n$-index} of~$X$. Note that any set~$X \in \M_n$ can be represented by infinitely many $M_n$-indices.

\begin{definition}
Let $c := (\Vec{R}, \sigma,X)$. A \emph{$\QQ_n$-index} of $c$ is a 3-tuple $\langle e_R,\sigma,e_X \rangle$ such that $Z_{e_X} = X$, i.e. $e_X$ is an $M_n$-index of $X$, and such that $\Phi_{e_R}^{\emptyset^{(n)}}(i,a,b) = R_i(a,b)$.
\end{definition}

In what follows, fix some~$n \geq 1$ and a set~$P$ of PA degree over~$\emptyset^{(n)}$ computing~$M_n$.

\begin{lemma}\label[lemma]{lem:effective-ownership-pr}
The statement \qt{$X \in \U^{\M_{n-1}}_{C_{n-1}}$} is $\Pi^0_1(C_{n-1} \oplus (X \oplus M_{n-1})')$
uniformly in~$X$. In particular, if~$X \in \M_{n-1}$, then it is $\Pi^0_1(C_{n-1} \oplus M_{n-1}')$ uniformly in an~$M_n$-index of~$X$.
\end{lemma}
\begin{proof}
%For all code $e$ and sets $X,Z$, the sentence \qt{$X \in \U_e^Z$} holds if $(\exists \rho \subseteq X) \rho \in W^Z_e$, hence is $\Sigma_1^0(X \oplus Z)$. 
The sentence 
 \qt{$X \in \U_{C_{n-1}}^{\M_{n-1}}$} is equivalent to the formula~$(\forall (e,i) \in C_{n-1})( X \in \U_e^{Z_i})$, where $M_{n-1} = \bigoplus_{i \in \omega} Z_i$. In other words, the sentence  \qt{$X \in \U_{C_{n-1}}^{\M_{n-1}}$} is equivalent to
 $$\forall e \forall i,\ (e, i) \not \in C_{n-1} \vee (\exists \rho \subseteq X) \rho \in W^{Z_i}_e
 $$
 The left-hand side of the disjunction is $\Delta^0_0(C_{n-1})$, and the right-hand side is $\Sigma^0_1(X \oplus M_{n-1})$, hence $\Delta^0_0((X \oplus M_{n-1})')$.
 The whole sentence is therefore $\Pi_1^0(C_{n-1} \oplus (X \oplus M_{n-1})')$ uniformly in $X$.
\end{proof}

Fix a set~$Z$ and a sequence of pairs of $\Pi^0_1(Z)$ formulas $(\varphi_s, \psi_s)_{s \in \NN}$ such that for every~$s$, at least one of $\varphi_s$ and $\psi_s$ is true. It is well-known that any set~$P$ of PA degree relative to~$Z$ computes a set~$H$ such that for every~$s$, if~$s \in H$ then $\varphi_s$ is true, and if $s \not \in H$ then $\psi_s$ is true. Combining this fact with \Cref{lem:effective-ownership-pr}, we obtain the following lemma:

\begin{lemma}\label[lemma]{lem:effective-pigeonhole}
Let~$X \in \M_n \cap \U^{\M_{n-1}}_{C_{n-1}}$ and $f : X \to 2$ be a 2-coloring in~$\M_n$.
Then there is some~$f$-homogeneous set~$Y \subseteq X$ such that $Y \in \M_n \cap \U^{\M_{n-1}}_{C_{n-1}}$. Moreover, an $M_n$-index of~$Y$ can be $P$-uniformly from $M_n$-indices of~$X$ and and~$f$.
\end{lemma}
\begin{proof}
Let~$\P_X$ be the class of all~$H$ such that for every~$i$, if~$i \in H$ then $X \cap Z_i \in \U^{\M_{n-1}}_{C_{n-1}}$, and if~$i \not \in H$, then $X \setminus Z_i \in \U^{\M_{n-1}}_{C_{n-1}}$. By partition regularity of $\U^{\M_{n-1}}_{C_{n-1}}$ and since~$X \in \U^{\M_{n-1}}_{C_{n-1}}$, then $\P_X$ is non-empty. Moreover, by \Cref{lem:effective-ownership-pr}, the class $\P_X$ is $\Pi^0_1(C_{n-1} \oplus M_{n-1}')$ uniformly in an $M_n$-index of~$X$, hence is $\Pi^0_1(\M_n)$.
Thus, given an $M_n$-index of~$X$, one can find an $M_n$-index of a tree~$T_X$ such that $[T_X] = \P_X$, and of a member~$H \in [T_X]$,
and given an $M_n$-index $i_f$ of~$f$, one can $H$-decide whether~$X \cap f^{-1}(0)$ or $X \cap f^{-1}(1)$ belongs to $\U^{\M_{n-1}}_{C_{n-1}}$ and thus compute an $M_n$-index of the corresponding set.

Note that almost all the operations are manipulations of codes, hence do not use the oracle~$P$. The only place it is used is when deciding whether~$X \cap f^{-1}(0)$ or $X \cap f^{-1}(1)$ belongs to $\U^{\M_{n-1}}_{C_{n-1}}$. Indeed, it requires to \qt{evaluate} the $M_n$-index of~$H$ into its actual set, thanks to the oracle~$M_n$ which is~$P$-computable.
\end{proof}

The following two lemmas analyze the uniformity of the forcing questions at the lower levels and at the top level.

\begin{lemma}\label[lemma]{lem:effective-question-below-validity}

In \Cref{lem:question-below-validity}, a $\QQ_n$-index of an extension~$d$ can be found~$P$-uniformly from a $\QQ_n$-index of~$c$ and the $\Sigma^0_{k+1}$ formula~$(\exists x)\psi_e(G, x)$ for~$k < n$.
\end{lemma}
\begin{proof}
Let $c := (\vec{R},\sigma,X)$ with $\QQ_n$-index~$\langle e_R, \sigma, e_X \rangle$.

   % First, suppose $n =0$.
    \begin{itemize}
        \item Suppose $\sigma \qvdash_m (\exists x) \psi_e(G,x)$. As in the proof of \Cref{lem:question-below-validity}, by a compactness argument, there exists $t \in \omega$ such that for every 2-colorings $\vec{h} = h_0, \dots, h_{m-1} \in 2^t$, there is a finite $\tau \subseteq X \cap \{0, \dots, t \}$ which is $\vec{R}$-transitive and $\vec{h}$-homogeneous and some $x \in \omega$ such that
        \begin{itemize}
            \item if $k = 0$, $\psi_e(\sigma \cup \tau,x)$,
            \item if $k > 0$, $\sigma \cup \tau \nqvdash_\ell\neg \psi_e(G, x)$ for some $\ell \geq m.$
        \end{itemize} 
        $M_n$-indices of the colorings $\vec{g}$ defined in \Cref{lem:question-below-validity} are uniformly $P$-computable in~$\langle e_R, \sigma, e_X \rangle$. By \Cref{lem:effective-pigeonhole}, one can $P$-compute uniformly in~$M_n$-indices of~$\vec{g}$ and $e_X$ an $M_n$-index of a $\vec{g}$-homogeneous set~$H \subseteq X$ in $\M_n \cap \U^{\M_{n-1}}_{C_{n-1}}$.

        The colorings $\vec{h}$ are defined uniformly from the colors of $\vec{g}$-homogeneity of $H$. Thus, the finite $\vec{R}$-transitive and $\vec{h}$-homogeneous set $\tau \subseteq X \cap \{0,\dots,t\}$ is found $P$-uniformly in~$\langle e_R, \sigma, e_X \rangle$. One can therefore $P$-compute a $\QQ_n$-index of~$d := (\vec{R},\sigma \cup \tau, H)$ uniformly from a $\QQ_n$-index of~$c$.

        If $k = 0$, $d$ is the desired extension.
        If $k > 0$, since~$\sigma \cup \tau \nqvdash_\ell\neg \psi_e(G, x)$, then by induction hypothesis, one can $P$-computably find a $\QQ_n$-index of an extension~$p \leq d$ such that $p \Vdash \psi_e(G, x)$, uniformly in a $\QQ_n$-index of~$d$, hence in a $\QQ_n$-index of~$c$.

    \item Suppose $\sigma \nqvdash_m (\exists x) \psi_e(G,x)$. 
    For every set~$Y$, let $\P_Y$ be the class of all $m$-tuple of 2-colorings $\vec{h} = h_0, \dots, h_{m-1} \in 2^Y$ and $\vec{S} \in \C_{m}(\sigma,Y)$, such that for all $\tau \subseteq Y \setminus \{0, \dots, |\sigma|\}$ which is $\vec{S}$-transitive and $\vec{h}$-homogeneous, and for all~$x \in \omega$, 
    \begin{itemize}
        \item if $k = 0$, $\neg \psi_e(\sigma \cup \tau,x)$,
        \item if $k > 0$, $\sigma \cup \tau \qvdash_\ell \neg \psi_e(G, x)$ for all $\ell \geq m$.
    \end{itemize}
    Note that the class~$\P_Y$ is $\Pi^0_1(M_k \oplus Y)$ uniformly in~$Y$.
    
    Since $\sigma \nqvdash_m (\exists x) \psi_e(G,x)$, there exists $s \in \omega$ and $Y_0, \dots Y_s$ a partition of $\omega$ such that for all $i \leq s$, $(\dagger)$ either $Y_i \not \in \U_{C_k}^{\M_k}$ or $\P_{Y_i} \neq \emptyset$. Let~$\L$ be the class of all such $s$-partitions of~$\omega$. Since the statement \qt{$Y_i \not \in \U_{C_k}^{\M_k}$} is $\Pi^0_1(C_k \oplus M_k')$ and the statement \qt{$\P_{Y_i} \neq \emptyset$} is $\Pi^0_1(M_k \oplus Y_i)$ uniformly in~$Y_i$, the class $\L$ is $\Pi^0_1(C_k \oplus M_k')$ and in particular is $\Pi^0_1(\M_n)$ since~$k < n$. One can $P$-compute uniformly in an index of~$\L$ an $M_n$-index of some~$(Y_0, \dots, Y_s) \in \L$. By \Cref{lem:effective-pigeonhole}, one can $P$-compute uniformly in an $M_n$-index of $(Y_0, \dots, Y_s)$ and an $M_n$-index of~$X$ some~$i \leq s$ and an $M_n$-index of some set~$H_0 \subseteq X \cap Y_i$ in $\M_n \cap \U^{\M_{n-1}}_{C_{n-1}}$. In particular, $Y_i \in \U_{C_k}^{\M_k}$, thus $\P_{X_i} \neq \emptyset$ and one can $P$-computably find an $M_n$-index of an $m$-tuple~$\vec{h}$ and $\vec{S} \in \C_m(\sigma, Y)$ in $\P_{X_i}$.
    
    By \Cref{lem:effective-pigeonhole}, one can $P$-compute the $M_n$-index of a $\vec{h}$-homogeneous subset~$H \subseteq H_0$ in~$\M_n \cap \U^{\M_{n-1}}_{C_{n-1}}$ uniformly in~$\langle e_R, \sigma, e_X \rangle$. Thus, a $\QQ_n$-index of the condition $d := (\vec{R}\vec{S}, \sigma, H)$ can be $P$-computed uniformly in a $\QQ_n$-index of~$c$ and the formula $(\exists x)\psi_e(G, x)$.

    \end{itemize}
\end{proof}

\begin{lemma}\label[lemma]{lem:effective-question-validity}
In \Cref{lem:question-validity}, a $\QQ_n$-index of an extension~$d$ can be found~$P$-uniformly from a $\QQ_n$-index of~$c$ and the $\Sigma^0_{n+1}$ formula~$(\exists x)\psi_e(G, x)$.
\end{lemma}
\begin{proof}

Let $m = |\vec{R}|+1$.

\begin{itemize}
    \item 
First, suppose $c \qvdash (\exists x) \psi_e(G,x)$. The proof is 
essentially the same as the first case of \Cref{lem:effective-question-below-validity}: One define colorings $\vec{g}$ and $\vec{h}$ similarly, and refine the reservoir into a $\vec{g}$-homogeneous subset thanks to \Cref{lem:effective-pigeonhole}.
One therefore obtains a $\QQ_n$-index of~$d := (\vec{R},\sigma \cup \tau, H)$ uniformly from a $\QQ_n$-index of~$c$, where $\sigma \cup \tau \nqvdash_\ell\neg \psi_e(G, x)$.

Then, by \Cref{lem:effective-question-below-validity}, one can $P$-computably find a $\QQ_n$-index of an extension~$p \leq d$ such that $p \Vdash \psi_e(G, x)$, uniformly in a $\QQ_n$-index of~$d$, hence in a $\QQ_n$-index of~$c$.

\item Now, suppose $c \nqvdash (\exists x) \psi_e(G,x)$. Then, there exists an $m$-tuple of 2-colorings $\vec{h} = h_0, \dots, h_{m-1} \in 2^X$ and an $m$-tuple of tournaments $\vec{S} \in \C_m(\sigma,X)$ such that for every $\vec{h}$-homogeneous and $\vec{S}$-transitive finite chain $\sigma \subseteq X$ and every $x \in \omega$, $\sigma \cup \tau \qvdash_m \neg \psi_e(G,x)$. Let $\L$ be the class of such $m$-tuples of 2 colorings $\vec{h}$ and $m$-tuples of $\vec{S} \in \C_m(\sigma,X)$. By \Cref{lem:approx-pi01}, $\C_m(\sigma, X)$ is $\Pi^0_1(X)$. Since $X \in \M_n$, the class~$\L$ is $\Pi_1^0(\M_n)$, hence, since $P \geq_T M_n$, one can $P$-compute $M_n$-indices of a pair $(\vec{h}, \vec{S}) \in \M_n \cap \L$. By \Cref{lem:effective-pigeonhole}, one can $P$-compute a $\QQ_n$-index of a $\vec{h}$-homogeneous set~$Y \subseteq X$ in~$\P$ uniformly in $\M_n$-indices of $X$ and $\vec{h}$.
The 3-tuple $d := (\vec{R}\vec{S}, \sigma, Y)$ is a valid $\EM$-condition such that $d \Vdash (\forall x) \neg \psi_e(G,x)$.

\end{itemize}

\end{proof}

We are now ready to prove \Cref{thm:effective-strong-main}.

\begin{proof}[Proof of \Cref{thm:effective-strong-main}]

Let $n \geq 1$. Let $P$ be a set of PA degree over $\emptyset^{(n)}$. Let $(\psi_s(G))_{s > 0}$ be an enumeration of all $\Sigma_{k+1}^0$ formulas for all $k \leq n$. 

Let us begin with a condition $ c_{0} := (T, \emptyset, \omega)$. By induction, suppose $c_{s-1}$ built, with $i_{s-1}$ a $\QQ_n$-index of $c_{s-1}$, Consider the $\Sigma^0_{k+1}$ formula $\psi_s(G)$ , for $k \leq n$. By \Cref{lem:effective-question-below-validity} if $k < n$ and by  \Cref{lem:effective-question-validity} if $k = n$, there exists an extension $d \leq c_{s-1}$ such that $d \Vdash \psi_s(G)$ or  $d \Vdash \neg \psi_s(G)$. Moreover, a $\QQ_n$-index $i_d$ of $d$ is $P$-computable uniformly from~$i_{s-1}$ and $s$. Let $c_s := d$, and $i_s := i_d$. The $(c_s)_{s \in \omega}$ sequence is a decreasing sequence of conditions such that $\F := \{ d \in \QQ_n : (\exists s)c_s \leq d \}$ is n-generic. By \Cref{lem:1-generic-infinite} and \Cref{lem:gf-transitive}, $G_{\F}$ is infinite and $T$-transitive, and by \Cref{lem:forcing-implies-truth}, $G_\F^{(n)} \leq_T P$.
\end{proof}

\bibliographystyle{plain}
\bibliography{biblio}

\end{document}